\documentclass[11pt,reqno]{amsart}
\usepackage[utf8]{inputenc}
\usepackage{hyperref}
\usepackage{xcolor}
\usepackage{amssymb,amscd}
\usepackage{latexsym}
\usepackage[all]{xy}
\usepackage{xspace}
\usepackage{mathrsfs}
\usepackage{setspace}
\usepackage{cases}
\usepackage{bbm}  
\usepackage{url}
\newtheorem{Thm}{Theorem}[section]
\newtheorem{Lem}[Thm]{Lemma}
\newtheorem{Cor}[Thm]{Corollary}
\newtheorem{Prop}[Thm]{Proposition}

\theoremstyle{definition}
\newtheorem{Rem}[Thm]{Remark}
\newtheorem{Def}[Thm]{Definition}
\newtheorem{Expl}[Thm]{Example}

\newcommand{\CC}{\mathbbm{C}} 
\newcommand{\NN}{\mathbbm{N}} 
\newcommand{\ZZ}{\mathbbm{Z}} 
\newcommand{\ZZp}{\ZZ_{\geq 1}}

\newcommand{\df}{\colon}
\newcommand{\id}{\operatorname{id}}

\newcommand{\cF}{{\mathcal F}}

\newcommand{\cM}{{\mathcal M}}

\newcommand{\cO}{{\mathcal O}}
\newcommand{\cP}{{\mathcal P}}

\newcommand{\fg}{\mathfrak{g}}

\newcommand{\fn}{\mathfrak{n}}

\newcommand{\bb}{\mathbf{b}}

\newcommand{\bd}{{\mathbf d}}
\newcommand{\be}{{\mathbf e}}

\newcommand{\bm}{{\mathbf m}}

\newcommand{\bp}{{\mathbf p}}
\newcommand{\bq}{{\mathbf q}}
\newcommand{\br}{{\mathbf r}} 
\newcommand{\bs}{{\mathbf s}}

\newcommand{\bv}{{\mathbf v}}
\newcommand{\bw}{{\mathbf w}}
\newcommand{\bx}{{\mathbf x}}

\newcommand{\sfA}{{\mathsf{A}}}
\newcommand{\tsfA}{{\widetilde{\mathsf{A}}}}
\newcommand{\sfB}{{\mathsf{B}}}
\newcommand{\sfC}{{\mathsf{C}}}
\newcommand{\tsfC}{{\widetilde{\mathsf{C}}}}

\newcommand{\alp}{\alpha}   
\newcommand{\bet}{\beta}

\newcommand{\del}{\delta}

\newcommand{\zet}{\zeta}
\newcommand{\tht}{\theta}
\newcommand{\lam}{\lambda}

\newcommand{\ome}{\omega}

\newcommand{\vep}{\varepsilon}

\newcommand{\Del}{\Delta}

\newcommand{\Ome}{\Omega}

\newcommand{\hgt}{\operatorname{ht}}
\newcommand{\rk}{\operatorname{rank}}

\newcommand{\lmd}{\operatorname{\!-mod}}

\newcommand{\rep}{\operatorname{rep}}
\newcommand{\replf}{\operatorname{rep}_{\mathrm{l.f.}}}

\newcommand{\Mat}{\operatorname{Mat}}

\newcommand{\dimv}{\underline{\dim}}
\newcommand{\rkv}{\underline{\rk}}
\newcommand{\rks}{\underline{\operatorname{rk}}}

\newcommand{\supp}{\operatorname{supp}}

\newcommand{\Hom}{\operatorname{Hom}}
\newcommand{\Mor}{\operatorname{Mor}}
\newcommand{\Ext}{\operatorname{Ext}}

\newcommand{\St}{\operatorname{St}}
\newcommand{\Ba}{\operatorname{Ba}}
\newcommand{\pBa}{\operatorname{pBa}}
\newcommand{\Stlf}{\St_{\mathrm{lf}}}

\newcommand{\bil}[1]{\langle #1\rangle}
\newcommand{\abs}[1]{\left| #1\right|}

\newcommand{\bsm}{\begin{smallmatrix}}
\newcommand{\esm}{\end{smallmatrix}}

\newcommand{\bbsm}{\left[\begin{smallmatrix}}
\newcommand{\besm}{\end{smallmatrix}\right]}

\newcommand{\bbm}{\begin{matrix}}
\newcommand{\ebm}{\end{matrix}}

\newcommand{\diag}{\operatorname{diag}}

\newcommand{\GL}{\operatorname{GL}}
\newcommand{\Gr}{\operatorname{Gr}}

\newcommand{\ra}{\rightarrow}

\newcommand{\real}{{\rm re}}
\newcommand{\imag}{{\rm im}}

\newcommand{\II}{\mathbbm{1}}

\begin{document}

\date{August 10, 2025}

\title[Geometric construction of $U(\fn)$ for affine type $\tsfC_n$]{A geometric construction of $U(\fn)$ for affine Kac-Moody algebras of
  type $\tsfC_{n}$}

\author{Alberto Castillo Gómez}
\author{Christof Geiss}
\address{Instituto de Matemáticas, UNAM,
  Ciudad Universitaria, 04510 CDMX, México}
\email{betho1027488@gmail.com}
\email{christof.geiss@im.unam.mx}

\begin{abstract}
  Inspired by the work of Geiss, Leclerc and Schröer~\cite{GLS3} we realize the enveloping algebra of the positive part of an affine Kac-Moody Lie algebra of Dynkin type $\tsfC_n$ as 
  a generalized composition algebra of constructible
  functions on the varieties of locally free representations of the corresponding 1-Iwanaga-Gorenstein algebra $H=H_\CC(C,D,\Ome)$  with minimal symmetrizer $D$.
  To this end, we exploit in several ways the fact that in this situation $H$ is a  string algebra.  
\end{abstract}

\maketitle

\setcounter{tocdepth}{1}
\numberwithin{equation}{section}
\tableofcontents

\parskip2mm

\section{Introduction}
\subsection{Context and main result}
Let $C\in\Mat_{I\times I}(\ZZ)$ be a generalized symmetrizable Cartan matrix with (left) symmetrizer $D$ and an acyclic orientation $\Ome$ of the diagram attached
to $C$. Consider the corresponding 1-Iwanaga-Gorenstein algebra
$H=H_\CC(C,D,\Ome)$ over the complex numbers from~\cite{GLS1},
and $\cM(H)=\oplus_{\br\in\NN^I}\cM_\br(H)$ the generalized composition algebra of constructible functions on the  varieties of locally free representations $\replf(H,\br)_{\br\in\NN^I}$ of $H$ from~\cite{GLS3}.
By definition,  $\cM(H)$ is generated by the characteristic functions 
$\theta_i$ of the generalized simple modules $E_i$ for $i\in I$. In fact, 
$\cM(H)$ is a Hopf algebra with $\theta_i\in\cM_{\alp_i}(H)$, 
where $\alp_i$ is the simple root corresponding to 
$i\in I$.  Since the generators
$\theta_i$ fulfill the (generalized) Serre relations,
we have a surjective algebra homomorphism
$\pi_H\df U(\fn(C))\ra\cM(H)$, where $U(\fn(C))$ is the universal enveloping algebra of the positive part $\fn(C)$ of the Kac-Moody
Lie algebra $\fg(C)$ attached to $C$, see~\cite[Thm.~1.1a]{GLS3}.  
It is well-known by the work of Schofield~\cite{Scf} and Lusztig~\cite{Lu}, that $\pi_H$ is an
isomorphism for $C$ symmetric and $D$ the trivial symmetrizer, since in this case $H$ is just the path algebra of a quiver.  The main result
of~\cite[Thm.1.1b]{GLS3} states that $\pi_H$ is also an isomorphism for each (possibly non-symmetric) Cartan matrix of finite type and any choice
of symmetrizer. The authors conjecture that $\pi_H$ is always an isomorphism.
Our main result adds evidence for this conjecture:

\begin{Thm}\label{thm:princ}
  Let $C$ be a generalized Cartan matrix of affine type $\tsfC_n$,
  and $D$ the minimal symmetrizer for $C$, then $\pi_H\df U(\fn(C))\ra\cM(H)$
  is an isomorphism for each orientation $\Ome$.  
\end{Thm}

Let $n\geq 2$ and $I=\{0, 1,\ldots, n\}$.
Recall, that the Cartan matrix resp.~the Dynkin 
diagram of affine type  $\tsfC_n$, is of the form
\[
C = \left(
\begin{matrix}
\phantom{-}2&-1&\phantom{-}0&&
\\
-2&\phantom{-}2&-1&\ddots&
\\
0&-1&\phantom{-}2&\ddots&\phantom{-}0 \\
&\ddots&\ddots&\ddots&-2
\\&& 0&-1&\phantom{-}2
\end{matrix}
\right)\in\ZZ^{I\times I},\quad\text{resp.}\quad
\def\objectstyle{\scriptstyle}
\def\labelstyle{\scriptstyle}
\xymatrix @-1.2pc{ 0\ar@{=}[r]|-{>}\ar@{=}[r] &1\ar@{-}[r] &2\ar@{-}[r]& \cdots
  &(n\!-\!2)\ar@{-}[r] &(n\!-\!1)\ar@{=}[r]|-{<}\ar@{=}[r] &n},
\]
and the minimal symmetrizer of $C$ is given by $D = \diag(2,1,\ldots,1,2)$. 
Recall, that  all imaginary roots are isotropic, and we have
$\Del^+_\imag(\tsfC_n)=\ZZ_{\geq 1}\rho$ for $\rho=(1,2,2,\ldots,2,1)$.
Any orientation of the Diagram is of the form
$\Ome=\{(i_1,i'_1), (i_2,i'_2),\ldots, (i_n,i'_n)\}$
with $\{i_k,i'_k\}=\{k-1,k\}$ for $k=1,2,\ldots, n$.

Henceforth, by a slight abuse of notation, we will write $\tsfC_n$ in place of $C$ whenever $C$ is of type $\tsfC_n$. 
Since $\tsfC_n$ is of affine type, the corresponding set of positive roots
$\Del^+(\tsfC_n)\subset \NN^I$ is  the union of the real roots
$\Del_\real^+(\tsfC_n)$ and 
the isotropic roots $\ZZ_{\geq 1}\rho$ for $\rho=(1,2,\ldots,2,1)$. Note, that $\Del_\real^+(\tsfC_n)\cap\ZZp\rho=\emptyset$.  

\subsection{Outline of the proof}
It is straight forward to see that the 1-Iwanaga-Gorenstein algebra 
$H = H_\CC(\tsfC_n, D,\Omega)$ is given by the quiver 
\[
Q(\tsfC_n,\Ome):\qquad \xymatrix{
  0 \ar@(ul,dl)_{\vep_0} & 1 \ar@{-}[l]_{\eta_1}& \cdots \ar@{-}[l]_{\eta_2}&
  n \ar@{-}[l]_{\eta_n}
\ar@(ur,dr)^{\vep_{n}}}
\]
with relations  $\vep_0^2=\vep_{n}^2=0$.  Here, the arrow $\eta_k$ points to the left if $(i_k, i'_k)=(k-1,k)$, else, it points to the right.  
Thus, $H$ is a representation-infinite string algebra in  the sense of~\cite{BR}.
Following~\cite{GLS1}, in this  situation
a representation $M$ of $H$ is locally free if $M(i)$ is free as a
$\CC[\vep_i]/(\vep_i^2)$ module for $i\in\{0, n\}$.   

From the combinatorial description of the indecomposable representations
of $H$ in terms of strings  and bands~\cite{BR} we derive, 
that the set of rank vectors of the indecomposable locally free
$H$-modules is precisely the set of positive roots $\Del^+(\tsfC_n)$. Moreover, we introduce an equivalence relation on the locally free strings such that $\bw'\in [\bw]$ implies canonically $M_{\bw'}(\eta_i)=M_\bw(\eta_i)$
for all $i\in\{1,2,\ldots,n\}$,  and in particular 
$\rkv(M_{\bw'})=\rkv(M_\bw)$.
It turns out that for each real root $\beta\in\Del^+_\real(\tsfC_n)$,
there exists a locally free string $\bw$ such that the string modules
$M_{\bw'}$ with $\bw'\in [\bw]$ represent all isomorphism classes of
indecomposable locally free $H$-modules $M$ with $\rkv(M)=\bet$.

In turn, the rank vector of each band module is an isotropic root. Moreover,
for each $k\in\ZZ_{\geq 1}$  there are several 1-parameter families of
indecomposable modules coming from bands, as well as several classes of
string modules
with rank vector $k\rho$. See Section~\ref{ssec:rklf} for more details.

The infinite dimensional, nilpotent Lie algebra $\fn:=\fn(\tsfC_n)$ is the positive part of the Kac-Moody Lie algebra $\fg:=\fg(\tsfC_n)$.  
It is $\ZZ^I$-graded, more precisely
$\fn=\oplus_{\alp\in\Del^+(\tsfC_n)}\fn_\alp$,
where $\Del^+(C)$ is the positive part of the corresponding root system.
Since $\tsfC_n$ is affine of untwisted type 
($\sfC_n^{(1)}$ in Kac's notation) we have
\[
\dim_\CC\fn_\alp=\begin{cases} 1 &\text{if } \alp\in\Del^+_\real(\tsfC_n),\\
                             n &\text{if } \alp=k\rho,\
\end{cases}
\]
see for example~\cite[Prp.~5.1 \& Cor.~7.4]{K1}).
On the other hand, the generalized composition algebra $\cM(H)$ can be
identified with the enveloping algebra
$U(\cP(H))$ of
the graded Lie algebra $\cP(H)\subset\cM(H)$ of
its primitive elements~\cite[Prp.~4.7]{GLS3}.
Thus, $\pi_H$ restricts to a surjective
homomorphism of graded Lie algebra $\fn(\tsfC_n)\ra\cP(H)$. 
This suggests the following strategy for the proof of Theorem~\ref{thm:princ}:

We show that for each real root $\bet\in\Del^+_\real(\tsfC_n)$, the characteristic function $\theta_\bet\neq 0$ of \emph{all} indecomposable locally free
$H$-modules $M$ with $\rkv(M)=\bet$ belongs to $\cP_\bet(H)$ and thus spans it.
It is slightly more complicated to find $n$ linearly independent elements
$(\tht_{k\rho}^{(i)})_{i=1,2,\ldots,n}$ for each isotropic root
$k\rho\in\Del_\imag^+(\tsfC_n)$.  It is however worth to mention that we can
choose the $\tht_{k\rho}^{(i)}$ such that only $\tht_{k\rho}^{(n)}$ contains
band modules in its support.  See Theorem~\ref{thm:primitive} for the exact
statement.

For the proof of this theorem, we introduce the defect 
$d_\Ome=\bil{-,\rho}_H$ on $\ZZ^I$, which allows us to divide the roots
$\Del_+(\tsfC_n)$ into three classes: preprojective ($d_\Ome(\bet)>0$),
regular ($d_\Ome(\bet)=0$) and preinjective ($d_\Ome(\bet)<0$).  For each class
we proceed by induction over the height to show
$\tht_\bet=[\tht_{\bet-\rho_j},\tht_{\rho_j}]$ for certain ``simple'' regular
roots $(\rho_j)_{j=1,2,\cdots,n}$. Finally, we set
$\tht^{(n)}_{k\rho}:=[\tht_{k\rho-\alp_i},\tht_i]$ and verify
$\tht_{k\rho}^{(n)}(M_{\bb,k,t})\neq 0$ for certain band modules $M_{(\bb,t,k)}$.

Our main tool to this end is N.~Haupt's~\cite{Haupt} description of the
multiplication of characteristic functions in $\cF(A)$ for an arbitrary
string algebra $A$.
He uses torus actions in order to express the evaluation of convolution
products in combinatorial terms, see Section~\ref{ssec:Haupt}.  

\subsection*{Conventions} We write $\NN:=\{0,1,2,\ldots\}$ for the natural
numbers, and $\ZZp:=\{1,2,3,\ldots,\}$ for the positive integers.


\section{Preliminary material}\label{sec1}

\subsection{Locally free modules}
Following~\cite[Def.~1.1]{GLS1}, in our situation with $H=H_\CC(\tsfC_n,D,\Ome)$,
an $H$-module $M$ is \emph{locally free}
if $M(i)$ is a free $\CC[\vep_i]/(\vep_i^2)$-module for $i\in\{0,n\}$.
If moreover all Auslander-Reiten translates $(\tau^k_HM)_{k\in\ZZ}$ are
locally free, we say that $M$ is $\tau$-\emph{locally free}.

By~\cite[Prp.~3.5]{GLS1} $M$ is locally free if and only if the projective
dimension of $M$ is at most 1, or equivalently the injective dimension
of $M$ is at most 1.  Thus, in particular, all projective and injective
$H$-modules are locally free. We  write $P_i:=He_i$ for the indecomposable
projective $H$-modules and $J_i:=\Hom_\CC(e_iH,\CC)$ for the indecomposable
injective $H$-modules ($i\in I$).

The  \emph{rank vector} $\rkv(M)=(\rk_{H_i} M(i))_{i\in I}$ of a locally free
$H$-module $M$ is given here by
\[
\rkv(M)_i=\begin{cases} \rk_{\CC[\vep_i]/(\vep_i^2)} M(i) &\text{ if } i\in\{0,n\},\\
\dim_\CC M(i) &\text{else.}
\end{cases}
\]
Denote by $(\alp_i)_{i\in I}$ the coordinate basis of $\ZZ^I$, and by
$E_i$ the unique locally free $H$-module with $\rkv(E_i)=\alp_i$.

Following Dlab and Ringel,  we introduce on $\ZZ^I$ a (non-symmetric) bilinear
form
\[
  \bil{-,-}_{H}\df\ZZ^I\times\ZZ^I\ra\ZZ,
\]
which is defined on the basis of simple roots $(\alp_i)_{i\in I}$ as follows:
\[
  r_{ij}=\bil{\alp_i,\alp_j}_H:=
\begin{cases}
\phantom{-}2 &\text{if } i=j\in\{0,n\},\\
\phantom{-}1 &\text{if } i=j\not\in\{0,n\},\\
    -2 &\text{if } (j,i)\in\Ome\text{ and } \{i,j\}\cap\{0,n\}\neq\emptyset,\\
    -1 &\text{if }  (j,i)\in\Ome\text{ and } \{i,j\}\cap\{0,n\}=\emptyset,\\
\phantom{-}0  &\text{else.}
\end{cases}
\]
Note, that with $R:=(r_{ij})_{i,j\in I}$ we have $R+R^t=D\tsfC_n$.
Moreover, $D^{-1}R\in\ZZ^{I\times I}$ is unitriangular up to a simultaneous permutation of rows and columns.
We have the following special case 
of~\cite[Prp.~4.1]{GLS1}:

\begin{Prop} \label{prp:ring}
  Let $M$ and $N$ be locally free $H$-modules, then we have
  \[
    \bil{\rkv(M),\rkv(N)}_H=\dim\Hom_H(M,N)-\dim\Ext_H^1(M,N).
\]
\end{Prop} 
Moreover, we have the following version of~\cite[Prp.~11.5]{GLS1}:
\begin{Prop}  \label{Prp:Cox}
  \begin{itemize}
  \item[(a)] With $c_H:=-R^{-1}R^t$ we have $c_H\cdot\rkv(P_i)=-\rkv(J_i)$ for all
    $i\in I$, i.e. $c_H$ is the Coxeter transformation for $H$.
  \item[(b)] Suppose the $M$ and its Auslander-Reiten translate $\tau_HM$
    are locally free and indecomposable, then we have
    \[ \rkv(\tau_H M)= c_H\cdot \rkv(M).\]
  \end{itemize}
\end{Prop}  
\begin{proof}
  (a)  By Proposition~\ref{prp:ring} we have
  \[
    c_i\del_{i,j}=\rkv(P_i)^t\cdot R\cdot\alp_j=\alp^t_i\cdot R\cdot \rkv(J_j)
  \]  
  for all $i,j\in I$, where we abbreviated $c_i:=D_{ii}$ for all $i$.
This implies easily our claim.

  (b) Follows from (a) by comparing~\cite[Prp.~11.5~(i)]{GLS1}
  with the discussion at the end of~\cite[Sec.~3.5]{GLS1}.  
\end{proof}  

\subsection{Roots and defect for  affine type $\tsfC_n$}\label{ssec:roots-tCn}
For $1\leq i\leq j\leq n$ we  set
\[
  \alp_{ij}:=\sum_{k=i}^j \alp_k,
\]
and for $1\leq i\leq j\leq n-1$ define
\[
  \bet_{ij}:= \alp_{in}+\alp_{j,n-1}.
\]
Thus, the $\alp_{ij}$ and $\bet_{ij}$ can be identified with the positive
roots for a finite root system of Dynkin type $\sfC_n$.

The minimal positive isotropic root for the affine type $\tsfC_n$ is
\[
  \rho:= (1,2,2,\ldots, 2,1)\in\Del^+_\imag(\tsfC_n)\subset\ZZ^I.
\]
Thus,  $\Del^+_\imag(\tsfC_n)=\ZZp\rho$.  It is well known, that
\begin{multline*}
\Del^+_\real(\tsfC_n)=  \{k\rho+\alp_{ij}\mid k\in\NN, 1\leq i\leq j\leq n\}
\textstyle{\coprod}
\{k\rho-\alp_{ij}\mid k\in\ZZp, 1\leq i\leq j\leq n\}\\
\!\!\textstyle{\coprod}
  \{k\rho+\bet_{ij}\mid k\in\NN, 1\leq i\leq j\leq n\!-\!1\}\textstyle{\coprod}
  \{k\rho-\bet_{ij}\mid k\in\ZZp, 1\leq i\leq j\leq n\!-\!1\},
\end{multline*}
since $\tsfC_n$ (i.e. $\sfC_n^{(1)}$ in Kac's notation) is one of the untwisted
affine types, see~\cite[Prp.~6.3]{K1}.  

With our notation, the roots of the form $k\rho\pm\bet_{ii}$ are the long ones,
all remaining real roots are short.

The linear from $d_\Ome\df\ZZ^I\ra\ZZ, \alp\mapsto\bil{\alp,\rho}_H$ is 
the \emph{defect}.   
A root $\alp\in\Del^+(\tsfC_n)$ is called \emph{preprojective},
resp.~\emph{regular},  resp.~\emph{preinjective} with respect to $\Ome$ if
$d_\Ome(\alp)>0$, resp.~$d_\Ome(\alp)=0$ resp.~$d_\Ome(\alp)< 0$.  

We leave it as an (easy) exercise to show that
\begin{equation} \label{eq:defect}
  R\cdot\rho =2\ (\!\sum_{\substack{i\in Q^o_0\\ i\text{ is a sink}}}\! \alp_i\ -
  \sum_{\substack{j\in Q^o_0\\ j\text{ is a source}}}\!\!\!\alp_j) ,
\end{equation}
where $Q^o=Q^o(\tsfC_n,\Ome)$ is obtained from $Q(\tsfC_n,\Ome)$ by deleting the
loops $\vep_0$ and $\vep_n$.

\subsection{Strings and bands for $H$}
\label{ssec:pre-string}
As mentioned in the introduction, our algebra $H=H_\CC(\tsfC_n,\Ome,D)$ is a string algebra in the sense of Butler and Ringel~\cite[Sec.~3]{BR}. 

Recall that for a quiver
$Q=(Q_0, Q_1, s, t)$  we denote by $Q_0$ the set of vertices, $Q_1$ the set
of arrows, and the functions $s, t\df Q_1\ra Q_0$ determine the start- and
terminal point of each arrow.  For $Q=Q(\tsfC_n,\Ome)$ we have $Q_0=I$ and
$Q_1=\{\vep_0, \eta_1,\ldots,\eta_n,\vep_n\}$. We have in our situation
$s(\vep_i)=i=t(\vep_i)$ for $i\in\{0,n\}$ and
$(t(\eta_j), s(\eta_j))=(i_j,i'_j)\in\Ome$ with the notation from the
introduction.

We define the set of letters
$Q_1^\pm:=Q_1\cup \{\bet^{-1}\mid\bet\in Q_1\}$, and extend the functions $s$ and $t$ to
$Q_1^\pm$ by $s(\bet^{-1}):=t(\bet)$ and $t(\bet^{-1})=s(\bet)$ for all
$\bet\in Q_1$. 
Moreover, we agree that $(\bet^{-1})^{-1}:=\bet$.
In our situation, the set $\St(H)$ of strings consists of the trivial words
$\II^{\pm 1}_{i}$ for $i\in I$ and the words $\bw=w_1 w_2\cdots w_l$ with
letters $w_j\in Q_1^\pm$ such that 
$s(w_{j-1})=t(w_j)$ and $w_{j-1}\neq w_j^{\pm 1}$ for $j=2,3,\ldots,l$. Thus,
the inverse $\bw^{-1}:= w_{l}^{-1}w_{l-1}^{-1}\cdots w_1^{-1}$ of a string
$\bw=w_1w_2\cdots w_l$ is also a string.

We agree that $s(\II_i^{\pm 1})=i=t(\II_i^{\pm 1})$ and $s(w):=s(w_l)$ whilst
$t(w)=t(w_1)$ for $w=w_1\cdots w_l$. We call the strings $\II_i$ and those
which consists only of letters from the set $Q_1$ \emph{direct strings},
and the inverses of direct strings are called \emph{inverse strings}.

Compatible strings can be concatenated:
if $\bv=v_1\cdots v_k$ with $s(\bv)=t(\bw)$ and
$v_k\neq w_1^{\pm 1}$ we have $\bv\cdot\bw:=v_1\cdots v_kw_1\cdots w_l$.
We  agree moreover
that $\II_{i-1}\eta_i=\eta_i\II_i=\eta_i$ if $(i-1,i)\in\Ome$ resp.
$\II_{i-1}\eta_i^{-1}=\eta_i^{-1}=\eta_i^{-1}\II_{i}$ if $(i,i-1)\in\Ome$, and
$\vep_0^{\pm 1}\II_0=\vep_0^{\pm 1}$ resp. $\II_n\vep_n^{\pm 1}=\vep_n^{\pm 1}$.

A \emph{band}  is a string
$\bb$ such that $\bb\!\cdot\!\bb$ is a string. A band $\bb$ is \emph{primitive} if
$\bb\neq \bv^n$ for any $\bv\in\St(H)$ and $n\geq 2$.
Note that if $\bb=b_1b_2\cdots b_l$ is a band, then also the \emph{rotation}
$\bb^{(1)}:=b_2b_3\cdots b_lb_1$ is a band. This allows to define recursively
$\bb^{(r)}$ for all $r\in\ZZ$. We denote by
$\Ba(H):=\{(\bb,*)\mid \bb\ \text{is a band}\}$  the set of bands for $H$,
and $\pBa(H)\subset\Ba(H)$ is the set of primitive bands.

\subsection{Windings}
Let $Q=(Q_0,Q_1, s_Q, t_Q)$ and $S=(S_0, S_1, s_S, t_S)$ be two quivers.
Following Krause~\cite{Krause} and Haupt~\cite[Sec.~2.3]{Haupt} a
\emph{winding} (for $Q$)  is a morphism of quivers $F\df S\ra Q$ 
such that
\begin{itemize}
\item[(s)]
  If $\alp, \bet\in S_1$ with $s_S(\alp)=s_S(\bet)$ and $\alp\neq\bet$, then
  $F(\alp)\neq F(\bet)$.
\item[(t)]
  If $\alp, \bet\in S_1$ with $t_S(\alp)=t_S(\bet)$ and $\alp\neq\bet$, then
  $F(\alp)\neq F(\bet)$.
\end{itemize}
If $F\df S\ra Q$ and $F'\df S'\ra Q$ are windings for $Q$, then a
\emph{morphism} from $F'$ to $F$ is a morphism of quivers
$\hat{F}\df S'\ra S$ such that
$F'=F\circ\hat{F}$.  We write $\Mor(F',F)$ for the set of all morphisms from
$F'$ to $F$, and note that the elements of $\Mor(F',F)$ are automatically
windings.  

If $(Q,I)$ defines a string algebra,  we say that a winding $F\df S\ra Q$ is
\emph{admissible} if moreover the following conditions hold:
\begin{itemize}
\item[(A)]
  The underlying graph of $S$ is of Dynkin type $\sfA$ or $\tsfA$.
\item[(P)]
  If $G$ is an automorphism of $F$  then $G=\id_S$.  
\item[(R)]
  There exists no path $a=\alp_1\alp_2\cdots\alp_l$ in $S$ with $F(a)\in I$. 
\end{itemize}

For each $\bw=w_1w_2\cdots w_l\in\St(H)$ we can define an admissible winding
  $F_\bw\df S(\bw)\ra Q(\tsfC_n,\Ome)$ as follows:  The underlying graph of $S(\bw)$ is of Dynkin type $\sfA_{l+1}$
  \[
    \xymatrix{1\ar@{-}[r]^{\alp_1} & 2\ar@{-}[r]^{\alp_2} &3\ar@{.}[rr] &&
      l\ar@{-}^>>>>{\alp_l}[r]&(l\!+\!1)}
    \]
 with the arrow $\alp_i$ pointing to the left if $w_i$ is a direct letter, and
 pointing to the right otherwise.  Moreover, for $i=1,2,\ldots l$ we set
 $F(\alp_i)=w_i$ if $w_i$ is a  direct letter and $F(\alp_i)=w_i^{-1}$ if
 $w_i$ is an inverse letter.  Finally, we set $F(i)=t(w_i)$ for $i=1,2,\ldots, l$ and $F(l+1)=s(s_l)$. 
 For the trivial words $\II_{i}^{\pm 1}$ we agree that
$S(\II_i^{\pm})$ is of Dynkin type $\sfA_1$, and $F_{\II_i}(1)=i$. 
 Note, that $F_\bw$ is isomorphic to $F_{\bw^{-1}}$.

 If $(\bb,*)\in\Ba(H)$ with $\bb=b_1b_2\cdots b_l$,
 we define a winding $F_{(\bb,*)}\df S(b,*)\ra Q(\tsfC_n,\Ome)$,
 with the underlying graph of $S(\bb,*)$ of type $\tsfA_{l-1}$
 
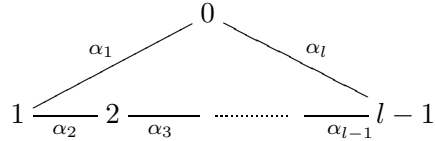
\[
 \xymatrix{&&\ar@{-}[lld]_{\alp_1}0 &&\\
  1\ar@{-}[r]_{\alp_2}&2\ar@{-}[r]_{\alp_3}&\ar@{.}[r]
  &\ar@{-}[r]_{\alp_{l-1}}&l-1\ar@{-}[llu]_{\alp_l}}     
\]
with the arrow $\alp_i$ pointing anti-clockwise if $b_i$ is a direct letter
and pointing clockwise otherwise.  Moreover, $F_{(\bb,*)}(\alp_i)=b_i$ if
$b_i$ is a direct letter and $\tilde{F}_{(\bb,*)}(\alp_i)=b_i^{-1}$ if $b_i$ is an inverse letter.  
Finally, $F_{(b,*)}(i)=t(b_{i+1})$ for $i=0,1,\ldots, l-1$.
Clearly, $F_{(b,*)}$ is admissible if and only if the band  $(b,*)$ is primitive.
Usually we visualize  a winding $F\df S\ra Q$  by displaying
the quiver $S$, where we label the vertices and arrows with the corresponding values of $F$. Note, that for $w\in\St(H)$ the
admissible windings $F_\bw$ and $F_{\bw^{-1}}$ are isomorphic. Similarly,
for $(\bb,*)\in\pBa(H)$ the admissible windings $F_{(\bb,*)}$,
$F_{(\bb^{(j)},*)}$ and $F_{(\bb^{-1},*)}$ are isomorphic.

It is easy to see, that each admissible winding for $Q(\tsfC_n,\Ome)$ is isomorphic to some $F_\bx$ with $\bx\in\St(H)\cup\pBa(H)$.

\subsection{String- and Band Modules}
If $F\df S\ra Q(\tsfC_n, \Ome)$ is an admissible winding, we obtain a 
push-forward functor $F^\lam\df\CC S\lmd\ra H\lmd$. See for example~\cite[Sec.~2.3]{Haupt} for details.

Let $S$ be a quiver of type $\sfA$, then we define the indecomposable
$\CC$-linear representation $\II_S$  of $S$ by taking
$\II_S(i)=\CC$ for all $i\in S_0$ and $\II_S(\alp)=1_\CC$ for each
arrow $\alp\in S_1$.  For $\bw\in\St(H)$  we define the representation
\[
  M_\bw:=F_\bw^\lam(\II_{S(\bw)}).
\]  
For $(\bb,*)\in\pBa(H)$, $t\in\CC^*$ and $m\in\ZZp$ we define the
indecomposable homogeneous regular $\CC$-linear representation
$R_{(\bb,*)}^{(t,m)}$ of $S(\bb,*)$ by
$R_{(\bb,*)}^{(t,m)}(i)=\CC^m$ for all vertices $i$ of $S(\bb,*)$, and
$R_{(\bb,*)}^{(t,m)}(\alp_1)=J_m(t)$ the indecomposable $m\times m$ 
Jordan block with generalized eigenvalue $t$, if $b_1$ is a direct letter.  Otherwise we set $R_{(\bb,*)}^{(t,m)}(\alp_1)=J_m(t^{-1})$.  
For the remaining arrows we have
$R_{(\bb,*)}^{(t,m)}(\alp_i)=\id_m$, the $m\times m$ unit matrix, regardless of the
orientation. This allows us to define the representation
\[
  M_{(\bb,t,m)}:=F_{(\bb,*)}^\lam(R_{(\bb,*)}^{(t,m)}).  
\]
It is well-known that the representations $M_\bw$ for $\bw\in\St(H)$ and
$M_{(\bb,t,m)}$ for $(\bb,*)\in\pBa(H)$ are indecomposable, and we have obvious
isomorphisms $M_\bw\cong M_{\bw^{-1}}$ as well as
$M_{(\bb,t,m)}\cong M_{(\bb^{(j)},t,m)}\cong M_{(\bb^{-1},t^{-1},m)}$.
Following Butler and Ringel~\cite[p.~161]{BR} each indecomposable $H$-module
is isomorphic to  $M_\bw$ for some $w\in\St(H)$ or to some $M_{(\bb,t,m)}$ for
some $(\bb,*)\in\pBa(H)$ and $(t,m)\in\CC^*\times \ZZp$.  Moreover,
the only isomorphisms between those string- and band modules are the ones
which come from the above discussed isomorphisms.

\subsection{Locally free strings}\label{ssec:locfree}
It follows from the definitions, that all band modules $M_{(\bb,t,m)}$ for
$(\bb,*)\in\pBa(H)$ and $(t,m)\in\CC^*\times\ZZp$ are locally free. 

We say that a string $\bw=w_1\cdots w_l\in\St(H)$ is \emph{locally free}
if $s(\bw)\in\{0,n\}$ implies $w_l=\vep_{s(w)}^{\pm 1}$ and $t(\bw)\in\{0,n\}$
implies
$w_1=\vep_{t(w)}^{\pm 1}$. We agree that also the strings $\II_i^{\pm 1}$ for
$i=1,2,\ldots, n-1$ are locally free.
Note that $\bw=w_1w_2\cdots w_l$ is locally free if and
only if $\bw^{-1}=w_l^{-1}w_{l-1}^{-1}\cdots w_1^{-1}$ is locally free.
We denote by $\Stlf(H)\subset\St(H)$ the corresponding
set of locally free strings.
It is easy to see that a string module $M_\bw$ for $\bw\in\St(H)$ 
is locally free if and only if $\bw$ is a locally free string.

We define for $\bx\in\Stlf(H)\cup\Ba(H)$ the rank vector
$\rks(\bx)=\operatorname{rk}(\bx)_{i\in I}\in\NN^I$   by
\begin{equation} \label{eq:rkv}
  \operatorname{rk}(\bx)_i:=\begin{cases} \abs{F_\bx^{-1}(i)}/2 &\text{ if } i\in\{0,n\},\\
                   \abs{F_\bx^{-1}(i)} &\text{else.}
   \end{cases}
  \end{equation}
We have then obviously
\[  
\rkv(M_\bx)=\rks(\bx)\quad\text{for all}\quad \bx\in\Stlf(H).
\]
On the the other hand, for $(\bb,*)\in\pBa(H)$ with $\bb=b_1b_2\cdots b_l$
we have $l=h(\bb)(2n+2)$ for some $h(\bb)\in\ZZp$, and it is easy to see
that
\[
  \rkv M_{(\bb,t,m)}=m\cdot\rks(\bb,*)=m\cdot h(\bb)\cdot (1,2,2,\ldots, 2,1)
\]
for  all $(t,m)\in\CC^*\times\ZZp$.

\subsection{Algebras of constructible functions} \label{ssec:AlgConstrF}
Let $Q$ be a quiver and $I\subset \CC Q$ an admissible ideal.  Then $H:=\CC Q/I$ is a finite dimensional basic $\CC$-algebra. 
Following~\cite[Sec.~10.19]{Lu}
we consider for each dimension vector $\bd\in\NN^{Q_0}$ the
$\CC$-vectorspace $\cF_\bd(H)$ of $G_\bd$-invariant constructible functions
on the affine representation variety $\rep(H,\bd)$ of $\bd$-dimensional
representations of $H$.  Here, the group 
$G_\bd:=\times_{i\in Q_0} \GL_\bd(\CC)$
acts on $\rep(H,\bd)$ by conjugation, such that the orbits correspond to
the isomorphism classes of representations of $H$ which have dimension vector $\bd$.

If $X\in\rep(A,\bd)$ we write $\cO(X):=G_\bd . X\subset\rep(A,\bd)$ for
its orbit. Since $\cO(X)$ is locally closed, the characteristic function
of the orbit $\chi_{\cO(X)}$ is an element of  $\cF_\bd(H)$. 

Let $H=\CC Q/I$ be a string algebra. For $\bw\in\St(H)$ we abbreviate
\[
  \chi_\bw:=\chi_{\cO(M_\bw)}
\]
for the characteristic function of the orbit of the string module $M_\bw$.
As just discussed, $\cO(M_\bw)\subset\rep(H,\dimv(M_\bw))$ is a locally closed
subset, and thus $\chi_\bw$ is a $G_\bd$-invariant, constructible function.

The graded vector space
\[
  \cF(H):=\bigoplus_{\bd\in\NN^{Q_0}} \cF_\bd(H)
\]
becomes a graded associative algebra with the usual convolution product
\[
  \phi'*\phi''(X):=\int_{Y\in\Gr^A_\bd(X)}\phi'(Y)\phi''(X/Y) d\!\chi,
\]
for $\phi'\in\cF_\bd(H)$ and $\phi''\in\cF_{\bd'}(H)$. 
Here, $\Gr^A_\bd(X)$ is the quiver Grassmannian
of $\bd$-dimensional subrepresentations of the representation
$X\in{\rep(H,\bd+\bd')}$. Moreover, our measure is given by the topological
Euler characteristic (with respect to Borel-Moore homology) $\chi$
for constructible sets, i.e.~we have
\[
  \int_{Y\in G} \psi(Y)d\!\chi :=\sum_{a\in\CC} a\cdot\chi(\psi^{-1}(a))
\]
for each constructible function $\psi\df G\ra\CC$. It is known that the map
\[
  c\df\cF(H)\ra\cF(H\times H) \text{ with } c(\psi)(X,Y):=\psi(X\oplus Y),
\]
is a  homomorphism of algebras of constructible functions,
see for example the proof of~\cite[Prp.~4.5]{GLS3}.  Note, that
$c(F_\bd(H))\subset \sum_{\bd'+\bd''=\bd}\cF_{\bd',\bd''}(H\times H)$.
We may view  $\cF(H)\otimes\cF(H)$ as a subalgebra of $\cF(H\times H)$
by setting $(\phi'\otimes\phi'')(X_1, X_2)=\phi'(X_1)\cdot\phi''(X_2)$. 

If $H=H_\CC(C,D,\Ome)$ is a GLS-algebra in the sense of~\cite{GLS1}, we may
consider the generalized composition algebra $\cM(H)$, 
namely the subalgebra of $\cF(H)$, 
which is generated by the
characteristic functions $\theta_i:=\chi_{\cO(E_i)}$ with $i\in I$, where
the modules $E_i$ are the generalized simple, locally free $H$-modules with rank vector
$\rk(E_i)=\alp_i$ corresponding to the simple roots. Since the functions
$\theta_i\in\cF_{\alp_i}(H)$ are homogeneous, $\cM(H)$ is also a graded algebra.
More precisely,  we have
\[
  \cM(H)=\bigoplus_{\br\in\NN^I} \cM_\br(H)
\]
where the support of functions from $\cM_\br(H)$ is contained in the
irreducible, open subvariety of locally free representations
$\replf(H,\br)\subset\rep(A,D\br)$, see~\cite[Sec.~4.2]{GLS3}.

The  morphism $c\df\cF(H)\ra\cF(H\times H)$ induces a
comultiplication $\Del\df\cM(H)\ra\cM(H)\otimes\cM(H)$, which gives
the space $\cM(H)$ the structure of
a cocommutative  bialgebra, see~\cite[Prp.~4.7]{GLS3}.  Recall, that
$\phi\in\cM(H)$ is \emph{primitive} if $\Del(\phi)=\phi\otimes 1+ 1\otimes\phi$.
In our situation, $\phi\in\cM_\br(H)$ is primitive if and only if its support
$\supp(\phi)\subset\replf(H,\br)$ consists only of \emph{indecomposable},
locally free representations, see~\cite[Lemma 4.6]{GLS3}. 

It is easy to see that the space of primitive elements $\cP(H)\subset\cM(H)$
becomes a Lie algebra with the usual commutator
$[\phi',\phi'']:=\phi'*\phi''-\phi''\phi'$.  In fact, the universal enveloping
algebra $U(\cP(H))$ is, as a bialgebra, isomorphic to $(\cM,*,\Del)$,
see~\cite[Prp.~4.7]{GLS3}.  

We have the following related result:
\begin{Lem} \label{lem:indec-mult}
  Let $X,Y\in A\lmd$ be indecomposable modules. Then  the support
\[
\chi_{\cO(X)} *\chi_{\cO(Y)} -\chi_{\cO(X\oplus Y)}\in\cF(A)
\]
consists only of indecomposable modules.  In particular, the same holds
for
\[
[\chi_{\cO(X)},\chi_{\cO(Y)}]:= \chi_{\cO(X)}*\chi_{\cO(Y)}-\chi_{\cO(Y)}*\chi_{\cO(X)}. 
\]
\end{Lem}

\begin{proof}
  Clearly, the first claim implies the second one. Let $Z:=Z_1\oplus Z_2$ with
  $\dimv Z=\dimv X+\dimv Y$ and $Z_1\neq 0\neq Z_2$. Consider $U\leq Z$
  with $U\cong X$ and $Z/U\cong Y$.  Then, $Z\not\cong X\oplus Y$ implies, 
  $U\not\subset Z_i$ for $i=1,2$.  However, this means that
  $U\in\Gr^A_{\dimv X}(Z)$  is \emph{not} a fixpoint under the 
  $\CC^*$-action on  $\Gr^A_{\dimv X}(Z)$, which is induced by the family of automorphisms of $Z$ which is given by 
  $(z_1,z_2)\mapsto (z_1,t z_2)$.  Thus,
  $\chi_{\cO(X)}*\chi_{\cO(Y)}(Z)=0$ by the argument in the proof
  of~\cite[Prp.~3.2]{DWZ2}.
\end{proof}

\subsection{Haupt's formula for string algebras} \label{ssec:Haupt}
The following result is a special case of~\cite[Cor.~3.17]{Haupt} adapted to
our notations.

\begin{Thm}[Haupt] \label{thm:Haupt}
  Let $I\subset \CC Q$ an admissible ideal such, that $H:=\CC Q/I$ is a string
  algebra. Let $\bv,\bw,\bx\in\St(H)$ and $(\bb,*)\in\pBa(H)$.  Then in
  the algebra of constructible functions $\cF$ the following holds:
  \begin{align*}
     (\chi_\bv *\chi_\bw)(M_\bx) &=
    \sum_{\substack{G\in\Mor(F_\bv, F_\bx)\\ H\in\Mor(F_\bw, F_\bx)}}
    (\chi_{G^\lam(\II_{S(\bv)})}* \chi_{H^\lam(\II_{S(\bw)})})(\II_{S(\bx)})\\
    (\chi_\bv *\chi_\bw)(M_{(\bb,t,m)}) &=
    \sum_{\substack{G\in\Mor(F_\bv, F_{(\bb,*)})\\ H\in\Mor(F_\bw, F_{(\bb,*)})}}
    (\chi_{G^\lam(\II_{S(\bv)})}* \chi_{H^\lam(\II_{S(\bw)})})(R_{S(\bb,*)}^{(t,m)})
  \end{align*}
\end{Thm}
Note, that for example $G^\lam(\II_{S(\bv)})$ is a string module for the
hereditary string algebra $\CC S(\bx)$ resp.~$\CC S(\bb,*)$.

\section{Rank vectors and roots}
\subsection{Basic locally free strings}\label{ssec:prplf}
In view of the definition of $\Ome$ it will be convenient to introduce  a
function
$\ome\df\{1,2,\ldots,n\} \ra\{-1,+1\}$ with $\ome(j)=+1$ iff
$(i_j,i'_j)=(j-1, j)$ i.e. if the arrow $\eta_j$ points to the left. Otherwise,
$\ome(j)=-1$.  Thus,
\begin{equation}
  \eta_{ij}:=\eta_{i+1}^{\ome(i+1)}\eta_{i+2}^{\ome(i+2)}\cdots\eta_j^{\ome(j)}\in\St(H)
\quad\text{with} (t(\eta_{ij}),s(\eta_{ij}))=(i,j)  
\end{equation}
for $0\leq i\leq j\leq n$, where we agree that $\eta_{ii}=\II_i$.
We also abbreviate $\eta:=\eta_{0,n}$.

For example, $\bb:=\eta\vep_n\eta^{-1}\vep_0\in\pBa(H)$  and
$\eta_{i,n}\vep_n\eta^{-1}\vep_0\eta_{0,i}$ is a rotation of $\bb$.

For each $i\in\{0,1,\ldots,n\}$ let $\bp_i^{(l)}$ be the longest direct
string with $s(\bp_i^{(l)})=i$, and $\bp_i^{(r)}$ the longest inverse string
with  $t(\bp_i^{(r)})=i$. Then $\bp_i:=\bp_i^{(l)}\bp_i^{(r)}\in\Stlf(H)$
and
$M_{\bp_i}\cong He_i  =: P_i$ is an indecomposable projective $H$-module
 with simple top $S_i$.  

Dually, for each $i\in\{0,1,\ldots,n\}$ let $\bq_i^{(l)}$ be the longest
inverse  string with $s(\bq_i^{(l)})=i$, and $\bq_i^{(r)}$ the longest direct
string with $t(\bq_i^{(r)})=i$. Then
$\bq_i:=\bq_i^{(l)}\bq_i^{(r)}\in\Stlf(H)$ and
$M_{\bq_i}\cong\Hom_\CC(e_iH, \CC) =: J_i$ is an indecomposable
injective $H$-module  with simple socle $S_i$.   
Moreover, we abbreviate for each $i\in I$
\begin{equation} \label{eqn:e_i}
\be_i:=\begin{cases}
\vep_0 &\text{if } i=0,\\
\II_i  &\text{if } 1\leq i\leq n-1,\\
\vep_n^{-1} &\text{if } i=n.
\end{cases}
\end{equation}
Thus, also $\be_i\in\Stlf(H)$ for all $i\in I$ and the modules $E_i:=M_{\be_i}$
are the generalized simple modules in the sense of~\cite[Sec.~3.2]{GLS1}.
Note, the if $i\in I$ is a  sink of $Q^o(\tsfC_n,\Ome)$ then $\bp_i=\be_i$, and
$\bq_i=\be_i$ if $i$ is a source of $Q^o(\tsfC_n,\Ome)$. 

For each $i\in I':=\{1,2,\ldots, n\}$ let
$\br_i$ be the longest inverse string such that the \emph{hook} 
$\eta_i\br_i$ is a string.  Note that some of the $\br_i$ may be of the
form $\II^{-1}_j$.
Similarly, let
$\br'_i$ be the longest direct string such that the \emph{co-hook}
$\eta_i^{-1}\br'_i$ is a string.  It is  easy to see, that
in our situation, there exists a $n$-cyclic 
permutation $\tau$ of $I'$  
such that $(\br'_i)^{-1}= \br_{\tau(i)}$ for all $i\in I'$, see for
example~\cite[Lemma 4.5.10]{Ricke} or~\cite[Prp.~3.8]{HLS21}.
In particular, $\br_i\eta_{\tau^{-1}(i)}\br_{\tau^{-1}(i)}$ is a string for all
$i\in I'$.
Similar strings and (co-) hooks are defined for the arrows $\vep_0$ and
$\vep_n$. However, these are not needed here, since our focus is on locally
free modules.
Using the maximality properties, it is straight forward to check that
\[
\{\bp_i\mid i\in I\} \cup \{\bq_i\mid i\in I\} \cup \{\br_i\mid i\in I'\}
\subset\Stlf(H).
\]
\begin{Expl} \label{expl:ct5}
Let $n=5$ and $\Ome=\{(0,1), (2,1), (3,2), (4,3), (4,5)\}$.  Thus we have:
\[
Q(\tsfC_n,\Ome):\qquad \xymatrix{
  0 \ar@(ul,dl)_{\vep_0} & \ar[l]_{\eta_1} 1\ar[r]^{\eta_2}  & 2\ar[r]^{\eta_3}
  &3\ar[r]^{\eta_4} &4& \ar[l]_{\eta_5} 5 
  \ar@(ur,dr)^{\vep_{5}}
}
\]
We find here:
\begin{align*}
  \bp_0 &=\vep_0\II_1^{-1}=\vep_0  &\bq_0 &=\eta_1^{-1}\vep_0^{-1}\eta_1 &
  \br_1 &=\eta_2^{-1}\eta_3^{-1}\eta_4^{-1} &\br'_1 &=\br_2^{-1}\\
  \bp_1 &=\vep_0\eta_1\eta_2^{-2}\eta_3^{-1}\eta_4^{-1} & \bq _1&=\II_1 &
  \br_2 &=\eta_1^{-1}\vep_0^{-1}  &\br'_2 &=\br_3^{-1}\\
  \bp_2 &=\II_2\eta_3^{-1}\eta_4^{-1}=\eta_3^{-1}\eta_4^{-1} &
  \bq_2 &=\eta_2^{-1}\II_2=\eta_2^{-1} &\br_3 &=\II_2^{-1} &\br'_3 &=\br_4^{-1}\\
  \bp_3 &=\eta_4^{-1}  &\bq_3 &=\eta_2^{-1}\eta_3^{-1} &\br_4 &=\II_3^{-1}&
  \br'_4&=\br_5^{-1}\\
  \bp_4 &=\II_4 & \bq_4 &=\eta_2^{-1}\eta_3^{-1}\eta_4^{-1}\eta_5\vep_5 &
  \br_5 &=\vep_5^{-1}\eta_5^{-1} &\br'_5&=\br_1^{-1}\\
  \bp_5 &=\eta_5\vep_5^{-1}\eta_5^{-1} &\bq_5 &=\vep_5.
\end{align*}
We see from the last column, that in this case the cyclic permutation
$\tau$ is given by
\[
  1 \mapsto 2 \mapsto 3 \mapsto 4 \mapsto 5 \mapsto 1
\]  
\end{Expl}

\subsection{Operations on locally free strings}
For each $\bw\in\Stlf(H)$ there exists a unique pair 
$(s'(\bw),s''(\bw))\in I'\times\{-1,1\}$ such that $\bw\eta_{s'(w)}^{s''(w)}$
is a string. Thus, $s''(\bw)=\ome(s'(\bw))$ with our notation from
Section~\ref{ssec:pre-string}.
We define moreover $(t'(\bw), t''(\bw)):=(s'(\bw^{-1}), s''(\bw^{-1}))$.

Suppose $s''(\bw)=+1$, then, in view of the definition, also
$\bw[1]:=\bw\eta_{s'(\bw)}\br_{s'(\bw)}$ is also a locally free string, and
we have $(s'(\bw[1]), s''(\bw[1]))=(\tau^{-1}(s'(\bw)),+1)$. Thus, 
we can define
recursively $\bw[n+1]:= (\bw[n])[1]$ for all $n\in\NN$.
If $s''(w^{-1})=+1$ we define $[n]w:= (w^{-1}[n])^{-1}$.
Similarly, if $s''(\bw)=-1$ the locally free string
$\bw[-1]:=\bw\eta_{s'(\bw)}^{-1}\br'_{s'(\bw)}$ is defined and we have
$(s'(\bw[-1]),s''(\bw[-1]))= (\tau(s'(\bw)),-1)$.
Thus, we can define in this situation
$\bw[-n]$ for all $n\in\NN$, as well as we can define $[-n]\bw$ in case $s''(\bw^{-1})=-1$.

The following result follows easily from the description of the prinjective
component of the Auslander-Reiten quiver of $H(\tsfC_n,D,\Ome)$
in~\cite[Sec.~4.5]{Ricke}.

\begin{Lem} \label{lem:proj}
  For each $i\in I$, either the vertex $t(\bp_i)$ is a sink of $Q^o$ and 
  $\bp_i=\be_{t(\bp_i)}[j]$ for some $j\in\NN$, or else $s(\bp_i)$ is a sink of
  $Q^o$ and $\bp_i=[j]\be_{s(\bp_i)}$ for some $j\in\NN$. 

  Similarly, either 
  $t(\bq_i)$ is a source in $Q^o$, and $\bq_i=\be_{t(\bq_i)}[-k]$ for some
  $k\in\NN$, or else $s(\bq_i)$ is a source of $Q^o$ and
  $\bq_i=[-k]\be_{s(\bq_i)}$ for some $k\in\NN$. 
\end{Lem}

\begin{Expl}\label{expl:ct5-2}
  Continuing with Example~\ref{expl:ct5}, we see that in this situation we
  find
  \begin{align*}
    \bp_0 &=\be_0=\vep_0,    &\bp_1 &=\be_0[1]=\vep_0\eta_1\br_1, &
    \bp_2 &=[2]\be_4=\II_2\eta_{3}^{-1}\II_3\eta_4^{-1}\II_4,\\
    \bp_3 &=[1]\be_4\II_3\eta_4^{-1}\II_4, &
    \bp_4 &=\be_4, &
    \bp_5 &=\be_4[1]=\II_1\eta_5\br_5.
  \end{align*}
Also, we find
  \[
    \br_1[4]=\br_1\eta_5\br_5\eta_4\br_4\eta_3\br_3\eta_2\br_2=
    (\eta_2^{-1}\eta_3^{-1}\eta_4^{-1})\eta_5(\vep_5^{-1}\eta_5^{-1})
    \eta_4\cdot\eta_3\cdot\eta_2(\eta_1^{-1}\vep_0^{-1}),
  \]
and thus, $\rks(\br_1[4])=\rho=(1,2,2,2,1)$.  
\end{Expl}

This motivates the following:

\begin{Def} \label{def:properties}
We say that a locally free string $\bw\in\Stlf(H)$ is
\begin{itemize}
\item  \emph{weakly preprojective} if $s''(\bw)=+1=t''(\bw)$,
\item  \emph{weakly regular} if $s''(\bw)\neq t''(\bw)$,  
\item  \emph{weakly isotropic} if it is weakly regular with $s'(\bw)=t'(\bw)$
\item  \emph{weakly preinjective} if $s''(\bw)=-1=t''(\bw)$.  
\end{itemize}
On the other hand, the locally free strings of the form
\begin{itemize}
\item
  $([k]\bp_i[k])^{\pm 1}$ for $(i,k)\in I\times\NN$ are called
  \emph{preprojective},
\item
  $(\br_i[k])^{\pm 1}$ for $(i,k)\in I'\times\NN$ are called
  \emph{regular},
\item
  $(\br_i[(k+1)n - 1])^{\pm 1}$ for $(i,k)\in I\times\NN$ are called
  \emph{isotropic}
\item
  $([-k]\bq_i[-k])^{\pm 1}$ for $(i,k)\in I\times\NN$ are called
  \emph{preinjective}.
\end{itemize}
The set $\Stlf^\tau(H)$ of $\tau$-\emph{locally free strings} consists,
by definition, of the  union of the preprojective, regular and preinjective
strings.
\end{Def}

\begin{Rem} \label{rem:tlf}
  (1) With the help of~\eqref{eq:defect} it is easy to see that a locally
  free string $\bw$ is weakly preprojective if and only if 
  $d_\Ome(\rks(\bw))>0$.  Similarly,  it is weakly regular if and only if
  $d_\Ome(\rks(\bw))=0$.  Else, $\bw$ is weakly preinjective.

  (2) Clearly, the ``projective'' strings $\bp_i$ for $i\in I$ are
  weakly preprojective, the ``simple regular'' strings $\br_i$ for $i\in I'$
  are weakly regular, and the ``injective'' strings $\bq_i$ for $i\in I$ are
  weakly preinjective.  On the other hand, if  $\bp$ is weakly
  projective, then also $\bp[1]$ and $[1]\bp$ are weakly projective.
  Thus, by induction, all preprojective strings are weakly preprojective,
  all regular strings are weakly regular, and all preinjective strings
  are weakly preinjective.  For the claim about isotropic strings, we
  observe that $t''(\br_i[k])=t''(\br_i)=+1$ and 
  $s''(\br_i[k])=s''(\br_{\tau^{-k-1}(i)})=-1$.  

  (3) Suppose that $\bp\in\Stlf(H)$ is weakly preprojective, then it follows
  from~\cite[p.~172]{BR} that $M_\bp$ is the Auslander-Reiten translate
  of $M_{[1]\bp[1]}$.  Thus, $\rks([1]\bp[1])=c_H^{-1}\cdot\rks(\bp)$, where
  $c_H$ is the Coxeter transformation from Proposition~\ref{Prp:Cox}.

  Similarly, if $\bq$ is weakly preinjective,
  $M_{[-1]\bq[-1]}$ is the Auslander-Reiten translate of $M_\bq$, and we have
  $\rks([-1]\bq[-1])=c_H\rks(\bq)$.

  Moreover,
  $M_{\br_{\tau(i)}}$ is the Auslander-Reiten translate of $M_{\br_i}$ for all
  $i\in I'$.
  
  (4) As discussed in~\cite[Chapter~4]{Ricke} and~\cite{HLS21} 
  a string module $M_\bw$ is $\tau$-locally free if and
  only if $\bw\in\Stlf^\tau(H)$.
\end{Rem}

\begin{Lem} \label{lem:pq-unique}
  The family
\[
(\rks([k]\bp_i[k])_{(i,k)\in I\times\NN} \cup
  (\rks[-k]\bq_i[-k])_{(i,k)\in I\times\NN}
    \]
    of vectors in $\NN^I$ consists of pairwise different elements (which are in
    fact real roots).
\end{Lem}

\begin{proof}  
  In view of Remark~\ref{rem:tlf}~(3) we have
$(\rks([k]\bp_i[k]))_{(i,k)\in I\times\NN}= (c_H^{-k}\rkv(P_i))_{(i,k)\in I\times\NN}$
and
$(\rks([-k]\bq_i[-k]))_{(i,k)\in I\times\NN}= (c_H^k\rkv(J_i))_{(i,k)\in I\times\NN}$.
Thus, our claim follows for example from~\cite[Lemma 2.1]{GLS1} since $\tsfC_n$
is not of finite type.
\end{proof}

\subsection{Equivalence classes of locally free strings}
\label{ssec:equivlf}
We introduce a second set of letters
\[
  Q_1^*:= (Q_1^{\pm}\setminus\{\vep_0^{\pm 1}, \vep_n^{\pm 1}\})
  \cup\{\vep_0^*, \vep_n^*\}
\]
with
$s(\vep_i^*)=i= t(\vep^*_i)$ and $(\vep_i^*)^{-1} =\vep_i^*$ for $i\in\{0,n\}$.
We can form with these letters a set $\St^*$ of strings with the same rules
as above.  The canonical projection $p\df Q_0^\pm\ra Q_0^*$ with
$p(\eta_i)=\eta_i$, $p(\eta_i^{-1})=\eta_i^{-1}i$ and $p(\vep_i^{\pm 1})=\eta_i^*$
induces a surjective map $p\df\St(H)\ra\St^*(H)$.  We say then, that two
strings $\bv, \bw\in\St(H)$ are \emph{similar} if $p(\bv)=p(\bw)$, in symbols
$\bv\in[\bw]:=p^{-1}(p(\bw))$.
The notions  ``locally free'', ``weakly preprojective'',
``weakly regular'' and ``weakly preinjective''  are stable under this
equivalence relation.

We also introduce bands $\Ba^*(H)$ and primitive bands $\pBa^*(H)$ with the
obvious adaptions from $\Ba(H)$ resp.~$\pBa(H)$.  Note, however, that
$\pBa^*(H)$ contains up to rotation a unique band
$\tilde{\bb}:=\vep_0^*\eta\vep_n^* \eta^{-1}$,
which consist of exactly $2n+2$ letters.  We have now a projection
$\tilde{p}\df\pBa(H)\ra\Ba^*(H)$ with $p(\bb',*)=((\tilde{b}^{(i)})^{h(\bb')},*)$. 

Note that $\rks(\bw')=\rks(\bw)$ for all $\bw'\in [\bw]$.  

\subsection{Rank vectors of locally free strings are roots}
\label{ssec:rklf}
Recall, that we defined in Section~\ref{ssec:locfree} the rank function
$\rks\df\Stlf(H)\ra\NN^I$ for locally free strings.
Following~\ref{ssec:roots-tCn} we
identify the positive roots $\Del^+(\tsfC_n)$ with a subset of $\NN^I$.
With the notation from equation~\eqref{eqn:e_i} we have for example
$\rks(\be_i)=\alp_i$ for all $i\in I$.

\begin{Thm} \label{thm:rk-root}
Consider the string algebra $H=H_\CC(\tsfC_n, D,\Ome)$, where $D$ is the minimal symmetrizer. Then we have the following:
  \begin{itemize}
  \item[(a)]  The image of $\rks$ is the positive part $\Del^+(\tsfC_n)$ of
    the corresponding affine root system of type $\tsfC_n$.
  \item[(b)] For each $\bw\in\Stlf(H)$ with $\rks(\bw)\in\Del^+_\real(\tsfC_n)$
    we have
    \[
    \rks^{-1}(\rks(\bw))=[\bw] \cup [\bw^{-1}].
     \]
     Moreover, we have in this situation $[\bw]=[\bw^{-1}]$ if and only if
     $\rks(\bw)$ is a long root.
   \item[(c)]
     For each $\bw\in\Stlf(H)$ with
     $\rks(\bw)\in\Del^+_\imag(\tsfC_n)=\ZZp\rho$ we have
     \[
     \rks^{-1}(\rks(\bw))=\coprod_{i\in I'} ([\br_i[nk-1]]\ \textstyle{\coprod}\  [\br_i[nk-1]^{-1}])
     \]
     for some $k\in\ZZp$. 
  \end{itemize}
\end{Thm}

\begin{proof}
  The rank function $\rks\df\Stlf(H)\ra\NN^I$ factors over the projection
  $p\df\Stlf(H)\ra\Stlf^*(H)$ from Section~\ref{ssec:equivlf}. Moreover,
  we observe that a string $\bx\in\Stlf^*(H)$ is uniquely  determined by
  its first letter and the number of letters.
  More precisely, we have precisely the following four types of locally free
  strings, where we abbreviate $I'=\{1,2,\ldots,n\}$ and
  $I'':=\{0,1,\ldots,n-1\}$. 
\begin{align*}
\bx_{i,j,k}^{-,+} &:= \begin{cases}
\eta_{i,j}(\eta_{j,n}\vep_n^*\eta^{-1}\vep_0^*\eta_{0,j})^k &\text{ if } i\leq j,\\
\eta_{i,n}\vep_n^*\eta^{-1}\vep_0^*\eta_{0,j}
(\eta_{j,n}\vep_n^*\eta^{-1}\vep_0^*\eta_{0,j})^k &\text{ if } i>j,
\end{cases}  &&\text{for } i\in I', j\in I'', k\in\NN,\\
\bx_{i,j,k}^{-,-} &:= \eta_{i,n}\vep_n^*\eta_{j,n}^{-1}
(\eta_{0,j}^{-1}\vep_0^*\eta\vep_n^*\eta_{j,n}^{-1})^k &&\text{for } i,j\in I', k\in\NN\\
\bx_{i,j,k}^{+,-} &:= \begin{cases}
  \eta_{j,i}^{-1}(\eta_{0,j}^{-1}\vep_0^*\eta^{+1}\vep_n^*\eta_{j,n}^{-1})^k
  &\text{ if } i\geq j,\\
  \eta_{0,i}^{-1}\vep_0^*\eta^{+1}\vep_n^*\eta_{j,n}^{-1}
  (\eta_{0,j}^{-1}\vep_0^*\eta^{+1}\vep_n^*\eta_{j,n}^{-1})^k &\text{ if } i<j,
\end{cases} &&\text{for } i\in I'', j\in I', k\in\NN,\\
\bx_{i,j,k}^{+,+} &:=
\eta_{0,i}^{-1}\vep_0^*\eta_{0,j}(\eta_{j,n}\vep_n^*\eta^{-1}\vep_0^*\eta_{0,j})^k
&&\text{for } i,j\in I'', k\in\NN.
\end{align*}  

It is straight forward to check, that
\begin{align*}
\rks(\bx_{i,j,k}^{-,+}) &=\begin{cases}
  \phantom{-}\alp_{i,j}+k\rho &\text{ if } i\leq j,\\
  \quad (k+1)\rho            &\text{ if } i = j+1,\\
  -\alp_{j+1,i-1}+(k+1)\rho    &\text{ if } i > j+1,
\end{cases}
 &&\text{for } i\in I', j\in I'', k\in\NN,\\
\rks(\bx_{i,j,k}^{-,-}) &= \begin{cases}
    \phantom{-}\bet_{i,j}+k\rho\hspace*{4.5em}  &\text{ if } i\leq j < n,\\
    \phantom{-}\alp_{i,n}+k\rho        &\text{ if } i\leq j = n,\\
    \phantom{-}\bet_{j,i}+k\rho        &\text{ if } n > i\geq j,\\
    \phantom{-}\alp_{j,n}+k\rho        &\text{ if } n = i\geq j,
  \end{cases}
  &&\text{for } i,j\in I', k\in\NN,\\
\rks(\bx_{i,j,k}^{+,-}) &=\begin{cases}
  \phantom{-}\alp_{j,i}+k\rho  &\text{ if } i\geq j,\\
  \quad (k+1)\rho            &\text{ if } i = j-1,\\
  -\alp_{i+1,j-1}+(k+1)\rho    &\text{ if } i < j-1,
\end{cases}
&&\text{for } i\in I'', j\in I', k\in\NN,\\
\rks(\bx_{i,j,k}^{+,+}) &= \begin{cases}
    -\bet_{i+1,j+1}+(k+1)\rho     &\text{ if } i\leq j < n-1,\\
    -\alp_{i+1,n}\;\;\ +  (k+1)\rho           &\text{ if } i\leq j = n-1,\\
    -\bet_{j+1,i+1}+ (k+1)\rho  &\text{ if } n-1 > i\geq j,\\
    -\alp_{j,n}\quad\;\,\ +  (k+1)\rho  &\text{ if } n-1 = i\geq j,
  \end{cases}
  &&\text{for } i,j\in I'', k\in\NN. \\
\end{align*}
Since we have moreover
\begin{align*}
(\bx_{i,j,k}^{+,-})^{-1} &=\bx_{j,i,k}^{-,+},&&(\bx_{i,j,k}^{-,-})^{-1}=\bx_{j,i,k}^{-,-},\\
(\bx_{i,j,k}^{-,+})^{-1} &=\bx_{j,i,k}^{+,-},&&(\bx_{i,j,k}^{+,+})^{-1}=\bx_{j,i,k}^{+,+},\\
\end{align*}
and in particular $(\bx_{i,i}^{-,-})^{-1}=\bx_{i,i}^{-,-}$ for $i\in I'$, as well as
$(\bx_{j,j}^{+,+})^{-1}=\bx_{j,j}^{+,+}$ for $j\in I''$,
the above calculations show claims (a) and (b). 

For (c) we observe first, that the above calculations show
\[
\{\bx\in\Stlf^*(H)\mid\rks(\bx)=k\rho\}=
\{\bx_{i-1,i,k}^{-,+} \mid i\in I'\}\cup \{\bx_{j,j+1,k}^{+,-} \mid j\in I''\}.
\]
Moreover, it is an easy exercise to verify, that
\[
p(\br_i[kn-1])=\begin{cases}
\bx_{i,i-1,k-1}^{-,+} &\text{ if } (i-1,i)\in\Ome,\\
\bx_{i-1,i,k-1}^{+,-} &\text{ if } (i,i-1)\in\Ome,
\end{cases}
\]
and in  particular
$\{p(\br_i[kn-1])^{\pm 1}\mid i\in I'\}=\{\bx\in\Stlf^*(H)\mid\rks(\bx)=k\rho\}$.
\end{proof}

\begin{Rem}
Recall that $\rho=(1,2,2,\ldots, 2,1)$ is the minimal  isotropic root
in $\Del^+(\tsfC_n)$. In view of our discussion in
Section~\ref{ssec:locfree}, the above result shows, that here
the set of rank vectors of \emph{all} indecomposable locally free
$H=H_\CC(\tsfC_n,\Ome,D)$-modules can be identified with $\Del^+(\tsfC_n)$.

In~\cite[Sec.~5.3]{GLS2} the authors conjectured that for any generalized
symmetrizable Cartan matrix $C$ with symmetrizer $D$ and any acyclic
orientation $\Ome$ the rank vectors of all $\tau$-locally free indecomposable
$H_K(C,D,\Ome)$-modules are in bijection with the positive roots $\Del^+(C)$.
They showed that this is true for $C$ of finite type, whilst already in
type $\sfB_3$ there exist indecomposable locally free modules $M$ with
$\rkv(M)$ \emph{not} a root.  In~\cite{HLS21} the above-mentioned conjecture is
verified for type $\tsfC_n$ with \emph{minimal} symmetrizer $D$,
by comparing rank vectors with the dimension
vectors of representations for a Dlab-Ringel species~\cite{DR} of type
$\tsfC_n$.   In our language, this means in particular that 
$\rks(\Stlf^\tau(H))=\Del^+(\tsfC_n)$.  In fact,  we have the following
more precise result, which will be useful.
\end{Rem}

\begin{Cor} \label{cor:root}
  Under the above hypothesis, for each $\bw\in\Stlf(H)$ there exists
  a $\bw_\tau\in\Stlf^\tau(H)$ such that 
  \[
    [\bw]\cap\Stlf^\tau(H)=\begin{cases} \{\bw_\tau, \bw_\tau^{-1}\} &\text{ if }
    \rks(\bw)\in\Del_{\real}^+ \text{ is a long root,}\\
    \{\bw_\tau\} &\text{ else.}
    \end{cases}
    \]
 More precisely, for each $\alp\in\Del^+_\real(\tsfC_n)$ we have
    \[
    \rks^{-1}(\alp)=\begin{cases}
        [([k]\bp_i[k])]\cup [([k]\bp_i[k])^{-1}] &
    \text{for a unique } (i,k)\in I\times\NN \text{ if } d_{\Ome}(\alp)>\,0,\\
             [(\br_i[k])]\cup [(\br_i[k])^{-1}] &
    \text{for a unique } (i,k)\in I'\!\times\NN \text{ if } d_{\Ome}(\alp)=0,\\
        [([-k]\bq_i[-k])]\cup [([-k]\bp_i[-k])^{-1}] &
        \text{for a unique } (i,k)\in I\times\NN \text{ if } d_{\Ome}(\alp) <\, 0.     
    \end{cases}
\]
Moreover, we have in this situation $ [([k]\bp_i[k])]=[([k]\bp_i[k])^{-1}] $
resp. $ [([-k]\bq_i[-k])]=[([-k]\bq_i[-k])^{-1} $ if and only if
  $\alp$ is a long root.
\end{Cor}

\begin{proof}
  Let $\bw\in\Stlf(H)$. By Theorem~\ref{thm:rk-root}~(a) we have
  $\rks(\bw)\in\Del^+(\tsfC_n)$. If $\rks(\bw)\in\Del^+_\imag(\tsfC_n)$,
  our claim follows from Theorem~\ref{thm:rk-root}~(c). Thus, we may assume
 $\rks(\bw)\in\Del^+_\real(\tsfC_n)$. In this case,
  by Theorem~\ref{thm:rk-root}~(b) and the above-mentioned result
  from~\cite{HLS21}, we find $\bw_\tau\in [\bw]\cap\Stlf^\tau(H)$. If moreover
  $\rks(\bw)=\rks(\bw_\tau)$ is preprojective, i.e. $d_\Ome(\rks(\bw_\tau))>0$,
  we may assume $\bw_\tau=[k]\bp_i[k]$ for some $(i,k)\in I\times\NN$ by the
  definition of $\Stlf^\tau(H)$, and our claim follows from
  Lemma~\ref{lem:pq-unique}. The case  $\rks(\bw)=\rks(\bw_\tau)$ preinjective
  is similar.

  If $d_\Ome(\rks(\bw_\tau))=0$, by the definition of $\Stlf^\tau(H)$  we may
  assume $\bw_\tau=\br_i[k]$ for some $(i,k)\in I'\times\NN$. Now
  $\bw'\in[\bw_\tau]\cap\Stlf^\tau(H)$ implies $\bw'=\br_{i'}[k']$ for
  some $(i',k')\in I'\times\NN$, since we need $t''(\bw')=t''(\bw_\tau)=+1$.
  By the same token we have $i'=i$, since  
  $t'(\bw')=t'(\br_i[k])=\tau(i)$.  This forces $k=k'$ for length reasons.
\end{proof}

\subsection{Factorization of locally free strings}
We have the following easy consequence of Corollary~\ref{cor:root}.
\begin{Cor}
  Let $\bw\in\Stlf(H)$ be weakly preprojective with
  $\rks(\bw)\not\in\{\alp_i\mid i\in I\}$, then there exists a weakly
  preprojective $\bv\in\Stlf(H)$ with $[\bv[1]]=[\bw]$.
\end{Cor}

  \begin{proof}
    By Corollary~\ref{cor:root} we may assume $[\bw]=[([k]\bp_i[k])]$ for
    some $(i,k)\in I\times\NN$. If $k\geq 1$ we can take $\bv=[k]\bp_i[k-1]$,
    which is also weakly preprojective.  If $k=0$ we can proceed with
    Lemma~\ref{lem:proj}. 
  \end{proof}

\begin{Rem} \label{rem:fact}
    Suppose that $\bv\in\Stlf(H)$ with $(s'(\bv),s''(\bv))=(i,+1)$ for some
    $i\in I'$, then we have
\[
    [\bv[1]]=\{\bv'\eta_i\br'\mid \bv'\in [\bv],\ \br'\in[\br_i]\},
\]
and moreover, these factorizations are unique.
\end{Rem}

\section{Main result}
\subsection{Characteristic functions}
Recall that $H=H_\CC(\tsfC_n,D,\Ome)$ with $D$ the minimal symmetrizer,
is our string algebra.
For each locally free string $\bw\in\Stlf(H)$ we define, with the notation
from Sections~\ref{ssec:prplf} and~\ref{ssec:AlgConstrF}, the
constructible function
\[
\chi_{[\bw]}:=\begin{cases}
\ \sum_{\bw'\in [\bw]} \chi_{\bw'} &\text{if } \bw^{-1}      \not\in [\bw],\\
\frac{1}{2}\sum_{\bw'\in [\bw]} \chi_{\bw'} &\text{if } \bw^{-1} \in [\bw],
 \end{cases}
\]
which belongs to $\cF_{D\cdot\rks(\bw)}(H)$, and
$\chi_{[\bw]}$ is supported on indecomposable representations in
$\replf(H,\rks(\bw))\subset\rep(H, D\cdot\rks(\bw))$. However,
at this stage it is not clear if $\chi_{[\bw]}\in\cP_{\rks(\bw)}(H)\subset\cM_{\rks(\bw)}(H)$.

\begin{Lem}
Let be $H=H(\tsfC_n,D,\Ome)$, then we have, with the just introduced notation,
the following:
\begin{itemize}
\item[(a)]     
  If $\alp\in\Del^+_\real(\tsfC_n)$ is a real root, there exists
  $\bw\in\Stlf(H)$, which is weakly isotropic, such that $\rks(\bw)=\alp$ and
  \[
  0\neq\Theta_\alp:=\chi_{[\bw]}\in\cF_{D\cdot\alp}
  \]
  is the characteristic function of the set of all indecomposable locally
  free representations $M$  of $H$ with $\rkv(M)=\alp$.
  In particular, $\Theta_\alp$ does not depend on the choice of $\bw$.
\item[(b)]
  For each $k\in\ZZp$ the functions $\sum_{i\in I'}\chi_{[\br_i[kn-1]]}$
  is the characteristic function of all locally free string modules $M$
  with $\rkv(M)=k\rho$.  In particular
  $(\chi_{\br_i}[kn-1])_{i\in I'}$ is a linearly independent family of
  constructible functions in $\cF_{k D\rho}(H)$.  
\end{itemize}
\end{Lem}

\begin{proof}
(a) Recall from the discussion in Section~\ref{ssec:locfree}, that the rank
vectors of all band modules are isotropic roots.
Thus, for $\alp$ a \emph{real} root, there exists no band module $M_{\bb,t,m}$
with  $\rkv(M_{\bb,t,m})=\alp$.
On the other hand, by Theorem~\ref{thm:rk-root}~(a) and~(b), in this situation
there exists a $\bw\in\Stlf(H)$ such that
 $\rks^{-1}(\alp)=[\bw]\cup [\bw^{-1}]$.
Now,   if $\bw^{-1}\not\in [\bw]$, we
have indeed $\{\bv^{-1}\mid \bv\in [\bw]\}=[\bw^{-1}]$ and
  $[\bw]\cap [\bw^{-1}]=\emptyset$.  Thus, in this case, the elements of
  $[\bw]$ represent all isomorphism classes of locally free representations with rank
vector $\rks(\bw)$ without repetition. 
If in turn $\bw^{-1}\in [\bw]$, the root $\rks(\bw)$ is
long, and $\bv\neq \bv^{-1}\in [\bw]$ for all $\bv\in [\bw]$. This means
that in this situation each indecomposable, locally free string module $M$
with $\rkv(M)=\alp$ is represented by exactly \emph{two} elements in
$[\bw]$.

Thus, $\chi_{[\bw]}$ is in both cases the
  characteristic function of \emph{all} indecomposable, locally free
  modules of $H$, which have rank vector $\alp$.

(b) The first claim follows from Theorem~\ref{thm:rk-root}~(c). The linear
independence follows immediately from this first observation.
\end{proof}

\subsection{Primitive elements}   
Let $H=H_\CC(\tsfC_n, D,\Ome)$ as before. Recall that the functions 
in the generalized composition algebra $\cM(H)\subset\cF(H)$ are supported by the locally free modules.  
The Lie algebra of primitive elements $\cP(H)\subset\cM(H)$  is
graded by the root lattice. More precisely, 
$\cP(H)=\oplus_{\alp\in\Del^+(\tsfC_n)}\cP_\alp(H)$, since we have $U(\cP(H))\cong\cM(H)$.
We are now ready to state our main result.
\begin{Thm} \label{thm:primitive}
With the above notation, the following holds:
\begin{itemize}
\item[(a)] For each real root $\alp\in\Del^+_\real(\tsfC_n)$  there exists
  a locally free string $\bw\in\Stlf(H)$, such that the characteristic function
  \[
    \Theta_\alp:=\chi_{[\bw]}\in\cF_{D\cdot\alp}(H)
    \]
  actually spans  $\cP_\alp(H)$.  
\item[(b)]
  For each $k\in\ZZp$ the constructible functions
\[
\Theta_{k\rho}^{(i)}:=\chi_{[(\br_i[kn-1])]} - \chi_{[(\br_{\tau(i)}[kn-1])]}
\in\cF_{D\cdot(k\rho)}(H)
\quad\text{for } i=1,2,\ldots,n-1,
\]
are linearly independent, and belong actually to $\cP_{k\rho}(H)$.
\item[(c)]
  For each $k\in\ZZp$ we have
  $k\rho-\alp_n\in\Del^+_\real(\tsfC_n)$, and with
  \[
  \Theta_{k\rho}^{(0)}:= [\Theta_{k\rho-\alp_j},\theta_j],
  \]
 the functions $\Theta^{(i)}_{k\rho}$ for $j=0,1,\ldots,n-1$ form a basis
  of $\cP_{k\rho}(H)$.
\end{itemize}
\end{Thm}

\begin{Rem}
(1)  For $i\in I$, and the corresponding simple root $\alp_i$, we have
  $\Theta_{\alp_i}:=\chi_{[\be_i]}=\chi_{\be_i}$.  
  Thus, in this case, there is nothing to show.
  This observation will be the starting point of our inductive proof of the theorem. In particular, we recall that the functions 
  $\Theta_i=\Theta_{\alp_i}$ for $i\in I$  are the generators of our generalized convolution algebra $\cM(H)$.  
    
(2) Note, that $\sum_{j=1}^n\Theta_{k\rho}^{(j)}=0$, if we define
$\Theta_{k\rho}^{(n)}:=\chi_{[(\br_n[kn-1])]} - \chi_{[(\br_{\tau(n)}[kn-1])]}$.
In particular, the family
$(\Theta_{k\rho}^{(j)})_{j=1,2,\ldots, n}$ is \emph{not} a basis of $\cP_{k\rho}(H)$.
\end{Rem}
  
We prove this theorem, after some preparation, in Section~\ref{ssec:proof}
Note, that this theorem implies immediately, that the map $\pi_H$ is an
isomorphism. 

\subsection{Key result for locally free string modules}
\begin{Prop} \label{prp:key}
  Let $\bw\in\Stlf(H)$ and $\br\in [\br_i]$ for some $i\in I'$,  then we have
  with $\zet:=\eta_{\tau^{-1}(i)}$ the following formula:
  \begin{multline*}
    [\chi_\bw,\chi_\br]=
     \del_{s''(\bw),+1}\del_{s'(\bw), i} \chi_{\bw\eta_{i}\br} 
    -\del_{s''(\bw),-1}\del_{s'(\bw),\tau^{-1}(i)} \chi_{\bw\zet^{-1}\br^{-1}}\\
    +\del_{t''(\bw),+1}\del_{t'(\bw), i}\chi_{\br^{-1}\eta_{i}^{-1}\bw}
    -\del_{t''(\bw),-1}\del_{t'(\bw),\tau^{-1}(i)} \chi_{\br\zet\bw}.
  \end{multline*}
\end{Prop}

\begin{proof}
  Recall that we can think of $\chi_\bw*\chi_\br$ and $\chi_\br*\chi_\bw$ as
  constructible functions on $\replf(H,\rks(\bw)+\rk(\bs))$. Thus, in view of
  Lemma~\ref{lem:indec-mult}, we only have to evaluate
  $\chi_\bw*\chi_\br$ and $\chi_\br*\chi_\bw$ on indecomposable locally free
  $H$-modules $M$ with $\rkv(M)=\rks(\bw)+\rks(\br)$.
  
  Suppose first that $\bw$ is weakly preprojective.  Then $d_\Ome(\rks(\bw))>0$,
  and since $d_\Ome(\rks(\br))=0$, we have trivially 
  $d_\Ome(\rks(\bw)+\rks(\br))>0$.
  In particular, $\rks(\bw)+\rks(\br)$ is not an isotropic root.  Thus, all
  indecomposable representations (if they exist at all) with rank vector
  $\rks(\bw)+\rks(\br)$, are  string modules $M_\bx$ with $\bx\in\Stlf(H)$ and
  $\rks(\bx)=\rks(\bw)+\rks(\br)$.  By Haupt's Theorem~\ref{thm:Haupt}
  and our hypothesis we find thus,
  \begin{align*}
    \chi_\bw*\chi_\br &=
    \del_{s'(\bw),i}\chi_{\bw\eta_i\br}+\del_{t'(\bw),i}\chi_{\br^{-1}\eta_i^{-1}\bw}+
        \chi_{\cO(M_\bw\oplus M_\br)},\\
        \chi_\br*\chi_\bw &= \chi_{\cO(M_\bw\oplus M_\br)}.
  \end{align*}
  This shows our claim in this case.
 
  Next, we study the case when $\bw$ is weakly regular.  Thus, we may
  assume $s''(\bw)=+1=-t''(\bw)$. In this case, for each $(\bb,*)\in\pBa(H)$,
  and $(t,m)\in\CC^*\times\ZZp$  we have
  \[
  \chi_{\bw}*\chi_\br(M_{(\bb,t,m)})=0=\chi_{\br}*\chi_\bm(M_{(\bb,t,m)})
  \]
  by Haupt's theorem. In fact, $\CC S(\bb,*)$ is a hereditary string algebra of
  type $\tsfA_l$ for $l=l(\bw)-1$.  Moreover, by our hypothesis, 
  for each $G\in\Mor(F_\br, F_{(\bb,*)})$ and
  $H\in\Mor(F_\bw, F_{(\bb,*)})$  the $\CC S(\bb,*)$-modules
  $G^\lam(\II_{S(\br)})$ and $H^\lam(\II_{S(\bw)})$ are \emph{regular} string
  modules. Thus, there are no homomorphisms to, or from those two modules, to
  the band module $R^{(t,m)}_{S(\bb,*)}$.
So, we have to deal in this case as well only with string modules and find,
  with a similar argument as in the previous case,
  \[
    [\chi_\bw,\chi_\br]=\del_{s'(\bw),i}\chi_{\bw\eta_i\br}
    -\del_{t'(\bw),\tau^{-1}(i)} \chi_{\br\zet\bw},
    \]
    which shows our claim for $\bw$ weakly regular.
    The argument for $\bw$ weakly preinjective is dual to the weakly
    preprojective case.
\end{proof}

\subsection{Simple weakly regular strings}
We call the locally free strings $\br$ with $\br\in [\br_i^{\pm 1}]$ for some
$i\in I'$ the \emph{simple weakly regular} strings.  We observe the following:
If $i<n$ and $(i-1,i)\in\Ome$, i.e.~if $\eta_i$ points to the left, we have
$\rks(\br_i)=\alp_{i,j}$ for some $j$ with $i\leq j\leq n$ since then
\[
\br_i=\begin{cases} \eta_{i,j}=\eta_{i+1}^{-1}\cdots\eta_j^{-1} &\text{if } j<n,\\
\eta_{i,n}\vep_n^{-1}=\eta_i^{-1}\cdots\eta_n^{-1}\vep_n^{-1} &\text{if } j=n.
\end{cases}
\]
Accordingly, we find
\[
[\br_i] =\begin{cases} \{\eta_{i,j}\} &\text{if } j<n,\\
          \{\eta_{i,n}\vep_n^{-1},\eta_{i,n}\vep_n\} &\text{if } j=n.
\end{cases}
\]
Similarly, for $(n-1,n)\in\Ome$ we have $\rks(\br_n)=\alp_{kn}$ for some
$k\in\{0,1,\ldots,n-1\}$, where we agree that $\alp_{0,n}=\rho-\alp_{1,n-1}$.
Indeed, we have in this situation
\[
\br_n=\begin{cases} \vep_n^{-1}\eta^{-1}_{k,n} &\text{if } k\geq 1,\\ 
\vep_n^{-1}\eta_{0,n}^{-1}\vep_0^{-1} &\text{if } k=0,
\end{cases}
\]
and accordingly
\[
  [\br_n]=\begin{cases} 
  \{\vep_n^{-1}\eta^{-1}_{k,n},\vep_n\eta^{-1}_{k,n}\} &\text{if } k\geq 1,\\ 
  \{\vep_n^{-1}\eta^{-1}_{0,n}\vep_0^{-1}, \vep_n^{-1}\eta^{-1}_{0,n}\vep_0,
  \vep_n\eta^{-1}_{0,n}\vep_0^{-1}, \vep_n\eta^{-1}_{0,n}\vep_0\} &\text{if } k=0.
  \end{cases}
\]  
The situation for $(i,i-1)\in\Ome$ is dual.
\begin{Lem} \label{lem:rho_i}
  For each $i\in I$, we have $\chi_{[\br_i]}\in\cP_{\rks(\br_i)}(H)$.
\end{Lem}
\begin{proof}
  Without loss of generality, we may assume $(i-1,i)\in\Ome$. If $i<n$ we
  have then $\rks(\br_i)=\alp_{i,j}$ for some $j\in\{i,i+1,\ldots, n\}$, 
  as discussed above.
  Since in this situation all the relevant (locally free) quiver Grassmannians
  are reduced to points, and we have only extensions in one direction, it
  is straightforward that we get
  \[
  \chi_{[\br_i]}=  [\cdots [[\theta_j,\theta_{j-1}], \theta_{j-2}],\cdots,\theta_i]
  \in\cP_{\alp_{i,j}}(H).
  \]
  In fact, this shows that $\chi_{[\br_i]}$ is in this case an iterated
  commutator of the primitive generators of $\cP(H)\subset\cM(H)$.
  Similarly, with $\rks(\br_n)=\alp_{k,n}$ we get
  \[
  \chi_{[\br_n]}=[\cdots [[\theta_k,\theta_{k+1}],\theta_{k+2}],\cdots,\theta_n]
  \in\cP_{\alp_{k,n}}(H).
  \]
\end{proof}

\subsection{Proof of the Main Theorem} \label{ssec:proof}
(a) Let $\bet\in\Del^+_\real(\tsfC_n)$. By Theorem~\ref{thm:rk-root} we
have $\rks^{-1}(\bet)=[\bw]\cup [\bw^{-1}]$ for some $\bw\in\Stlf(H)$.

Depending on the sign of the defect, namely $d_{\Ome}(\bet)>0$, 
$d_{\Ome}(\bet)=0$ or $d_{\Ome}(\bet)<0$,
we will proceed in each case by a slightly different 
induction on $\hgt(\bet)=\sum_{i\in I}\bet_i$.

 For $d_{\Ome}(\alp)>0$, we note first, that our claim holds trivially if  $\alp=\alp_i$ for some $i\in I$ which is a sink in $Q^o$. 
If however $\hgt(\bet)>1$, we find a weakly preprojective
$\bv\in\Stlf(H)$ with $[\bv[1]]=[\bw]$ and
$\hgt(\rk(\bv))<\hgt(\alp)$ by Corollary~\ref{cor:root}.    
Thus, we may assume by induction, that our claim
holds for $\bet=\rks(\bv)$. In particular, $\CC\chi_{[\bv]}=\cP_\bet(H)\neq 0$.
We abbreviate $s'(\bv)=i$ and $t'(\bv)=j$, and have then
$\bv[1]=\bv\eta_i\br_i$.  Now, $\CC\chi_{[\br_i]}=\cP_{\rks(\br_i)}(H)$ by
Lemma~\ref{lem:rho_i}, and thus
\[
0\neq [\chi_{[\bv]},\chi_{[\br_i]}]\in\cP_\bet(H)
\]
To be more precise, we have to distinguish two cases:
If $\bet$ is a short root, we have $i\neq j$, since $\bv$ is weakly preprojective. Thus,
\[
  [\chi_{[\bv]},\chi_{[\br_i]}]=
  \sum_{\substack{\bv'\in [\bv]\\ \br\in [\br_i]}} [\chi_{\bv'},\chi_{[\br]}]= 
  \sum_{\substack{\bv'\in [\bv]\\ \br\in [\br_i]}} \chi_{\bv'\eta_i\br}=\chi_{[\bw]} 
\]
by Proposition~\ref{prp:key} and Remark~\ref{rem:fact}.
If $\bet$ is a long root, $j=i$, and so
\[
  [\chi_{[\bv]},\chi_{[\br_i]}]=
  \sum_{\substack{\bv'\in [\bv]\\ \br\in [\br_i]}} [\chi_{\bv'},\chi_{[\br]}]= 
  \sum_{\substack{\bv'\in [\bv]\\ \br\in [\br_i]}}
  (\chi_{\bv'\eta_i\br}+\chi_{\br^{-1}\eta_i^{-1}\bv'}) = 2\chi_{[\bw]},
\]
since in this case $[\bv]=\{(\bv')^{-1}\mid \bv'\in [\bv]\}$.

In view of Corollary~\ref{cor:root}, it is sufficient to show,
for the regular case $d_\Ome(\bet)=0$, that
\begin{multline} \label{eqn:reg}
  [\cdots [\chi_{[\br_i]},\chi_{[\br_{\tau^{-1}(i)}]},\chi_{\br_{\tau^{-2}(i)}}],\cdots,
    \chi_{[\br_{\tau^{-k}(i)}]}] \\=\begin{cases}
  \chi_{[\br_i[k]]}                     &\text{if } k\not\equiv -1 \mod{n}\\
  \chi_{[\br_i[k]]}- \chi_{[\br_{\tau(i)}[k]]}&\text{if } k    \equiv -1 \mod{n},
  \end{cases}
\end{multline}
since, by Lemma~\ref{lem:rho_i}, $\chi_{[\br_i]}\in\cP_{\rks(\br_i)}(H)$ for all $i\in I'$. In fact, for all $(k,i)\in\NN\times I'$, we have
\[
(t'(\br),t''(\br))=(t'(\br_i[k]),t''(\br_i[k]))=(t'(\br_i),t''(\br_i))=(i,-1)
\]
for all $\br\in [\br_i[k]]$, and an easy induction shows that
\[
(s'(\br),s''(\br))=(a'(\br_i[k]),s''(\br_i[k]))=(\tau^{-k-1}(i),+1)
\]
for all $\br\in [\br_i[k]]$. Thus, our claim follows, by induction, from
Proposition~\ref{prp:key} and Remark~\ref{rem:fact}, as above.

The (weakly) preinjective case $d_\Ome(\alp)<0$ is dual to the (weakly) preprojective case.

(b) We saw already in~\eqref{eqn:reg}, that the functions
$\Theta^{(i)}_{k\rho}=\chi_{[(\br_i[kn-1])]} - \chi_{[(\br_{\tau(i)}[kn-1])]}$
are indeed primitive elements in $\cP_{k\rho}(H)$.  We note, that
$t(\br_i[kn-1])=s(\eta_i)\neq t(\eta_i)=s(\br_i[kn-1])$ for all
$(i,k)\in I'\times\ZZp$.  Thus,
$[\br_i[kn-1]] \cap [\br_j[kn-1]]=\emptyset$ if $i\neq j$. 
Moreover, for $\br\in [\br_i[kn-1]]$ we have $\br^{-1}\not\in [\br_i[kn]]$ for all $i\in I'$.  
We conclude, that the functions $(\chi_{\br_i[kn-1]})_{i\in I'}$
are linearly independent, and as a consequence the family
$(\Theta_{k\rho}^{(i)})_{i=1,2,\ldots,n-1}$ is linearly independent.

(c) Up to duality, we may assume that $n\in Q_0^o$ is a source, and we 
fix $k\in\ZZp$. 
Consider  the ``stable'' band $\bb=\eta^{-1}\vep_0^{-1}\eta\vep_n^{-1}$, i.e.
$(\bb,*)\in\pBa(H)$.
Recall that $\bet_k:=k\rho-\alp_n$ is a long real (preprojective)  root,
and note that
\[
\bb_k:= \eta_{0,n-1}^{-1}(\vep_0^{-1}\eta\vep_n^{-1} \eta^{-1})^{k-1}
        \vep_0^{-1}\eta_{0,n-1}
\]
is a weakly preprojective string with $[\bb_k]=\rks^{-1}(\bet_k)$.
On the other hand, clearly
$\Theta_{k\rho}^{(0)}:=[\Theta_{\bet_k},\Theta_{\alp_i}]\in\cP_{k\rho}(H)$, since
$\Theta_{\bet_k}\in\cP_{\bet_k}(H)$ by Part~(a),
and $\Theta_{\alp_i}=\Theta_i\in\cP(H)$ by definition.

Thus, it will be sufficient to show that
\[  
\Theta_{k\rho}^{(0)}(M_{\bb,t,k})\neq 0 \quad\text{for all } t\in\CC^*.  
\]
In view of Haupt's theorem, this is equivalent to
show
\[
\sum_{\bx\in [\bb_k]} \sum_{\substack{G\in\Mor(F_\bx, F_{(\bb,*)})\\ H\in\Mor(F_{\vep_n}, F_{(\bb,*)})}}
    [\chi_{G^\lam(\II_{S(\bx)})}, \chi_{H^\lam(\II_{S(\vep_n)})}]_{S(\bb,*)}(R_{S(\bb,*)}^{(t,m)})
    \neq 0.
\]
Now, by the construction of $\bb$, the set $\Mor(S(\vep_n), S(\bb,*))$ contains
a single element, say $E_n$. Moreover, $E_n^\lam(\II_{S(\vep_n)})$ is a preinjective
representation of $S(\bb,*)$, since $n$ is, by hypothesis, a source of $Q^o_0$.
On the other hand, $\Mor(S(\bx), S(\bb,*))=\emptyset$ for
$\bx\in[\bb_k]\setminus\{\bb_k,\bb_k^{-1}\}$, whilst $\Mor(S(\bb_k), S(\bb,*))$
contains an unique element $G_k$, up to isomorphism. Moreover,
$G_k^\lam(\II_{S(\bb_k)})$ is a preprojective representation of $S(\bb,*)$.
In particular,
\[
\chi_{H^\lam(\II_{S(\vep_n)})}*\chi_{G^\lam(\II_{S(\bb_k)})}=
\chi_{H^\lam(\II_{S(\vep_n)})\oplus G^\lam(\II_{S(\bb_k)})}.
\]
Thus, it is sufficient to verify
\[
\chi_{G^\lam(\II_{S(\bb_k)})}*\chi_{H^\lam(\II_{S(\vep_n)})} (R_{S(\bb,*)}^{(t,k)})=1.
\]
In fact, if we describe $\Hom_{S(\bb,*)}(G^\lam(\II_{S(\bb_k)}), R_{S(\bb,*)}^{(t,k)})$
with the help of the main theorem from~\cite[p.~191]{Krause}, the above analysis
shows in particular, that there is a unique admissible triple in the sense
of Krause. Thus, we have
\begin{equation} \label{eq:hom-tA}
\Hom_{S(\bb,*)}(G^\lam(\II_{S(\bb_k)}), R_{S(\bb,*)}^{(t,k)})\cong
  \Hom_\CC(\II_{S(\bb_k)}(1), R_{S(\bb,*)}^{(t,k)}(1))\cong \Hom_\CC(\CC,\CC^k).
\end{equation}
It is an easy exercise to see, that the corresponding homomorphism of
$S(\bb,*)$-representations is injective
if and only if the image of the map in the second term of~\eqref{eq:hom-tA}
is \emph{not}
contained in the maximal regular submodule of  $R_{S(\bb,*)}^{(t,k)}$, i.e. in
$R_{S(\bb,*)}^{(t,k-1)}(1)\cong\CC^{k-1}$. Thus, the space of subrepresentations of $R_{S(\bb,*)}^{(t,k)}$ which are isomorphic to the preprojective module
$G^\lam(\II_{S(\bb_k)})$, can be identified with $\CC^{k-1}$.  The corresponding
quotient is for all those submodules isomorphic to the generalized simple
module $E_n$. It just remains to recall that the topological Euler
characteristic  $\chi(\CC^{k-1})$ of the affine space $\CC^{k-1}$ is $1$.

\subsection*{Acknowledgment}
Alberto Castillo Gómez recognizes support via a PhD grant from CONACYT in the period
from 08/2017-07/2021.  He also recognizes partial support from Daniel
Labardini's Marcos Moshinsky Chair. Christof Geiss recognizes partial support from DGAPA-UNAM grant IN116723.



\begin{thebibliography}{99999}

\bibitem[BR87]{BR}
M.C.R.  Butler,  C.M. Ringel: 
\emph{Auslander-Reiten sequences with few middle terms and applications to string algebras}.
Comm.  Algebra \textbf{15} (1987),  no.  1–2,  145–179.


\bibitem[DeWZ10]{DWZ2}
H. Derksen, J. Weyman, A. Zelevinsky:
\emph{Quivers with potentials and their representations II: 
Applications to cluster algebras.}
J. Amer. Math. Soc. \textbf{23}  (2010), 749-790. 

\bibitem[DlR76]{DR}
V. Dlab,  C.M. Ringel:
\emph{Indecomposable representations of graphs and algebras}.
Mem. Amer.  Math.  Soc. (1976),  no.  173.

\bibitem[GLS17]{GLS1}
C.  Geiß,  B.  Leclerc,  J.  Schröer:
\emph{Quivers with relations for symmetrizable Cartan matrices I: Foundations}.
Invent.  Math.  209 (2017),  61-158.

\bibitem[GLS18]{GLS2}
C. Geiß, B. Leclerc, J. Schröer:
\emph{Quivers with relations for symmetrizable Cartan matrices II: Change of Symmetrizers}.
Int. Math. Res. Not. IMRN 2018, no. 9, 2866-2898.

\bibitem[GLS16]{GLS3}
C. Geiß, B. Leclerc, J. Schröer:
\emph{Quivers with relations for symmetrizable Cartan matrices III: Convolution algebras}.
Represent. Theory 20 (2016), 375-413.

\bibitem[Hau12]{Haupt}
N. Haupt: 
\emph{Euler characteristics of quiver Grassmannians and Ringel-Hall algebras of
  string algebras.}  
Algebr. Represent. Theory \textbf{15} (2012), no. 4, 755–793.

\bibitem[HLS23]{HLS21}
Hua-Lin Huang, Zengqiang Lin, Xiuping Su:
\emph{Components of AR-quivers for string algebras of type $\tsfC$ and a conjecture by Geiss-Leclerc-Schröer}. 
J. Algebra \textbf{632} (2023), 331–362. 

\bibitem[Kr91]{Krause}
H. Krause: 
\emph{Maps between tree and band modules}.  
J. Algebra  137  (1991),  no. 1, 186--194.

\bibitem[Kac90]{K1}
V. G. Kac:
\emph{Infinite-dimensional Lie algebras}.
Third Edition, Cambridge University Press,  (1990).


\bibitem[Lu91]{Lu}
G. Lusztig:
\emph{Quivers, perverse sheaves, and quantized enveloping algebras}.
J. Amer.~Math. Soc. \textbf{4} (1991), no. 2, 365–421

\bibitem[Rc16]{Ricke}
Ch. Ricke:
\emph{On $\tau$-tilting theory and perpendicular categories}.
Ph.D. Thesis, Bonn (2016)
\url{https://d-nb.info/1122193823/34}

\bibitem[Sc]{Scf}
A. Schofield:
\emph{Quivers and Kac-Moody Lie algebras}. Unpublished manuscript, 23pp.

\end{thebibliography}
\end{document}